\documentclass[a4j,12pt]{article} 
\usepackage{amsmath,amssymb}
\usepackage{mathdots}   
\newcommand{\bfd}{{\bf d}}
\newcommand{\bfb}{{\bf b}}

\allowdisplaybreaks[1]   

\DeclareMathOperator{\Ht}{ht}
\DeclareMathOperator{\Diag}{Diag}

\newcommand{\Section}[1]{%
\renewcommand{\thesection}{\S\arabic{section}}
\section{#1}
\renewcommand{\thesection}{\arabic{section}}
\setcounter{equation}{0}}

\newcommand{\qed}{\hfill \hbox{\rule[-2pt]{4pt}{7pt}}}
\newcommand{\proof}{{\hspace*{0.4cm} {\it Proof}.\ \enskip}}

\newtheorem{lem}{Lemma}[section]
\newtheorem{thm}[lem]{Theorem}  
\newtheorem{prop}[lem]{Proposition}
\newtheorem{cor}[lem]{Corollary}
\newtheorem{rem}[lem]{Remark}

\newcommand{\Z}{{\Bbb Z}}
\newcommand{\R}{{\Bbb R}}
\newcommand{\C}{{\Bbb C}}

\newcommand{\calH}{{\cal H}}

\newcommand{\calO}{{\cal O}}

\newcommand{\calR}{{\cal R}}
\newcommand{\calS}{{\cal S}}
\newcommand{\calU}{{\cal U}}

\newcommand{\pair}[2]{\left\langle {#1}, \, {#2}\right\rangle}

\newcommand{\set}[2]{\left\{\left.#1\vphantom{#2}\:\right\vert\:#2\right\}}
\newcommand{\wt}{\widetilde}
\newcommand{\what}{\widehat}

\newcommand{\fra}{{\frak a}}

\newcommand{\frp}{{\frak p}}

\newcommand{\lam}{{\lambda}}
\newcommand{\Lam}{{\Lambda}}
\newcommand{\ve}{{\varepsilon}}
\newcommand{\alp}{{\alpha}}  
\newcommand{\vphi}{{\varphi}} 
\newcommand{\eps}{{\epsilon}} 

\newcommand{\slit}{\vspace{5mm}}
\newcommand{\mslit}{\vspace{3mm}}
\newcommand{\real}{{\rm Re}}

\newcommand{{\any}}{{}^\forall}
\newcommand{{\is}}{{}^\exists}
\newcommand{{\st}}{\; {\rm s.t.}\; }

\newcommand{\ol}[1]{\overline{#1}}

\newcommand{\hec}{{\calH(G,K)}}

\newcommand{\SKX}{{\calS(K \backslash X)}}

\newcommand{\ch}{{\mbox{ch}}}

\newcommand{\abs}[1]{\left\vert{#1}\right\vert}  

\newcommand{\dint}[1]{\displaystyle{\int_{#1}}}

\newcommand{\dprod}[1]{\displaystyle{\prod_{#1}}}

\newcommand{\shita}[2]{\stackrel{\scriptstyle{#1}}{#2}}

\newcommand{\twomatrix}[4]{\begin{pmatrix}
                           {#1} & {#2}\\
                           {#3} & {#4}
                          \end{pmatrix}}

\newcommand{\twomatrixplus}[4]{\left(\begin{array}{c|c}
                           {#1}  & {#2}\\
                           \hline
                           {#3} &  {#4}
                          \end{array}\right)}

\newcommand{\mapright}[1]{\displaystyle{
   \smash{\mathop{\hbox to 1cm{\rightarrowfill}}^{#1}}}}

\newcommand{\gyaddots}{%
\setlength{\unitlength}{1mm}
\begin{picture}(5,3.5)(-2,-1.5)
\put(0,0){$\cdot$}
\put(2,1.5){$\cdot$}
\put(-2,-1.5){$\cdot$}
\end{picture}}
\newcommand{\gyakuddots}{\smash{\lower0.3ex\hbox{\gyaddots}}}

\begin{document}
\title{Spherical functions on the space of \\$p$-adic unitary hermitian matrices II, \\
the case of odd size }
\author{Yumiko Hironaka\\
{\small Department of Mathematics, }\\
{\small Faculty of Education and Integrated Sciences, Waseda University}\\
{\small Nishi-Waseda, Tokyo, 169-8050, JAPAN},\\
{\small and}\\
{Yasushi Komori}\\
{\small Department of Mathematics, }\\
{\small Faculty of Science, Rikkyo University}\\
{\small Nishi-Ikebukururo, Tokyo, 171-8501, JAPAN}\\
}

\date{}

\maketitle

\hfill \today
\begin{abstract}
We are interested in the harmonic analysis on $p$-adic homogeneous spaces based on spherical functions. In the present paper, we investigate the space $X$ of unitary hermitian matrices of odd size over a $\frp$-adic field of odd residual characteristic, which is a continuation of our previous paper where we have studied for even size matrices. 
First we give the explicit representatives of the Cartan decomposition of $X$ and introduce a typical spherical function $\omega(x;z)$ on $X$. After studying the functional equations, we give an explicit formula for $\omega(x;z)$, where Hall-Littlewood polynomials of type $C_n$ appear as a main term, though the unitary group acting on $X$ is of type $BC_n$. By spherical transform, we show the Schwartz space $\SKX$ is a free Hecke algebra $\hec$-module of rank $2^n$, where $2n+1$ is the size of matrices in $X$, and give parametrization of all the spherical functions on $X$ and the explicit Plancherel formula on $\SKX$. 
\end{abstract}

\bigskip
\begin{minipage}{15cm}


\end{minipage}

\setcounter{section}{-1}
\Section{Introduction}


 
We have been interested in the harmonic analysis on $p$-adic homogeneous spaces based on spherical functions. We have considered the space of $p$-adic unitary hermitian matrices of even size in \cite{HK}, and in the present paper we will study that of odd size.  The results can be formulated in parallel, though the groups acting the spaces have different root structures, $C_n$ for even and $BC_n$ for odd.  
The spaces have a natural close relation to the theory of automorphic functions and classical theory of sesquilinear forms (e.g. \cite{Oda}, \cite{Arakawa}). 

We fix an unramified quadratic extension $k'/k$ of $\frp$-adic field $k$ such that $2 \notin \frp$, and consider hermitian and unitary matrices with respect to $k'/k$, and denote by $a^*$ the conjugate transpose of $a \in M_{mn}(k')$.  Let $\pi$ be a prime element of $k$ and $q$ the cardinality of the residue class field $\calO_k/(\pi)$ and we normalize the absolute value on $k$ by $\abs{\pi} = q^{-1}$.

Denote by $j_m \in GL_m(k)$ the matrix whose all anti-diagonal entries are $1$ and others are $0$. Set
\begin{eqnarray*}
&&
G = U(j_m) = \set{g\in GL_m(k')}{gj_mg^* = j_m}, \quad K = G(\calO_{k'})\\
&&
X = \set{x \in X}{x^* = x, \; \Phi_{xj_m}(t) = \Phi_{j_m}(t)},\\
&&
g \cdot x = gxg^*, \quad (g \in G, \; x \in X),
\end{eqnarray*}
where $\Phi_y(t)$ is the characteristic polynomial of matrix $y$. We note that $X$ is a single $G(\ol{k})$-orbit over the algebraic closure $\ol{k}$ of $k$. 
Set
\begin{eqnarray*}
n = \left[\frac{m}{2}\right].
\end{eqnarray*}
According to the parity of $m$, $G$ has the root structure of type $C_n$ for even $m$ and type $BC_n$ for  odd $m$.
It is known in general that the spherical functions on various $p$-adic groups $\Gamma$ can be expressed in terms of the specialization of Hall-Littlewood polynomials of the corresponding root structure of $\Gamma$ (cf.\ \cite[\S 10]{Mac}, also \cite[Theorem 4.4]{Car}).  For the present space $X$, the main term of spherical functions can be written by using Hall-Littlewood polynomials of type $C_n$ with different specialization according to the parity of $m$ (cf.\ Theorem~3 below).

Let us note the results of even size in \cite{HK} and odd size in the present paper simultaneously, so that one can compare the results. 
In both papers, we study the Cartan decomposition and $G$-orbit decomposition in \S1, functional equations of spherical functions in \S2, explicit formulas of spherical functions in \S3, and harmonic analysis on $X$ in \S4. 

\slit
\noindent
{\bf Theorem~1} (1) {\it A set of complete representatives of $K \backslash X$ can be taken as}
\begin{eqnarray} \label{intro-Cartan}
\set{x_\lam}{\lam \in \Lam_n^+},
\end{eqnarray}
{\it where}
\begin{eqnarray*} 
&&
x_\lam = \left\{\begin{array}{ll}
\Diag(\pi^{\lam_1}, \ldots, \pi^{\lam_n}, \pi ^{-\lam_n}, \ldots, \pi^{-\lam_1})  & \textit{if }\; m=2n\\
\Diag(\pi^{\lam_1}, \ldots, \pi^{\lam_n}, 1, \pi ^{-\lam_n}, \ldots, \pi^{-\lam_1})  & \textit{if }\; m= 2n+1,
\end{array} \right. \\
&&
\Lam_n^+ = \set{\lam \in \Z^n}{\lam_1 \geq \lam_2 \geq \cdots \geq \lam_n \geq 0},
\end{eqnarray*}

\noindent
(2) {\it There are precisely two $G$-orbits in $X$ represented by $x_0 = x_{\bf0} = 1_m$ and $x_1 = x_{(1, 0, \ldots, 0)}$.}

\slit
The proof of Cartan decomposition for odd size needs a more delicate calculation than that for even size.  
If $k$ has even residual characteristic, there are some $K$-orbits without any diagonal matrix besides the above types, independent of the parity of the size. 

\mslit
A spherical function on $X$ is a $K$-invariant function on $X$ which is a common eigenfunction with respect to the convolutive action of the Hecke algebra $\hec$, and a typical one is constructed by Poisson transform from relative invariants of a parabolic subgroup.  We take the Borel subgroup $B$ consisting of upper triangular matrices in $G$.  
For $x \in X$ and $s \in \C^n$, we consider the following integral
\begin{eqnarray} \label{def-sph0}
\omega(x; s) = \int_{K} \prod_{i=1}^n \abs{d_i(k\cdot x)}^{s_i} dk,
\end{eqnarray}
where $d_i(y)$ is the determinant of the lower right $i$ by $i$ block of $y$, \; $1 \leq i \leq n$. 
Then the right hand side of (\ref{def-sph0}) is absolutely convergent for $\real(s_i) \geq 0, \; 1 \leq i \leq n$, and continued to a rational function of $q^{s_1}, \ldots , q^{s_n}$. Since $d_i(x)$'s are  relative $B$-invariants on $X$ such that 
\begin{eqnarray*}
d_i(p \cdot x) = \psi_i(p)d_i(x), \quad \psi_i(p) = N_{k'/k}(d_i(p)) \quad (p \in B, \; x \in X,\; 1 \leq i \leq n),
\end{eqnarray*}
we see $\omega(x; s)$ is a spherical function on $X$ which satisfies
\begin{eqnarray*}
&&
f * \omega(x; s) = \lam_s(f) \omega(x; s), \quad f \in \hec\\
&&
\lam_s(f) = \int_{B}f(p)\prod_{i=1}^n\abs{\psi_i(p)}^{-s_i} \delta(p)dp,
\end{eqnarray*}
where $dp$ is the left invariant measure on $B$ with modulus character $\delta$.  
We introduce the new variable $z \in \C$ related to $s$ by
\begin{eqnarray}
&&
s_i = -z_i + z_{i+1} -1 + \frac{\pi\sqrt{-1}}{\log q}, \quad 1 \leq i \leq n-1 \nonumber \\
&&
\label{eq:chgv}
s_n = \left\{\begin{array}{ll} 
     -z_n -\frac12+ \frac{\pi\sqrt{-1}}{\log q} & \mbox{if }m = 2n \\[2mm]
     -z_n -1+ \frac{\pi\sqrt{-1}}{2\log q} & \mbox{if } m = 2n+1, 
\end{array}\right.
\end{eqnarray}
and denote $\omega(x;z) = \omega(x;s)$ and $\lam_s = \lam_z$.  

The Weyl group $W$ of $G$ relative to $B$ acts on rational characters of $B$, hence on $z$ and $s$ also. The group $W$ is generated by $S_n$ which acts on $z$ by permutation of indices and by $\tau$ such that $\tau(z_1, \ldots, z_n)= (z_1, \ldots, z_{n-1}, -z_n)$. We will give the functional equation of $\omega(x;z)$ with respect to $W$. To describe the results we prepare some notation. 
We set 
\begin{eqnarray*}
&&
\Sigma^+ = \Sigma_s^+ \sqcup \Sigma_\ell^+,\\
&&
\Sigma_s^+ = \set{e_i + e_j, \; e_i -e_j}{1 \leq i < j \leq n}, \quad 
\Sigma_\ell^+ = \set{2e_i}{1 \leq i  \leq n}, 
\end{eqnarray*}
where $e_i \in \Z^n$ is the $i$-th unit vector, and define a pairing
\begin{eqnarray*}
\pair{\;}{} : \Z^n \times \C^n \longrightarrow \C, \; \pair{\alp}{z} = \sum_{i=1}^n \alp_iz_i.
\end{eqnarray*}

\slit
\noindent
{\bf Theorem~2} 
{\rm (1)} \textit{For any $\sigma \in W$, one has}
$$
\omega(x; z) = \dprod{\alp}\, 
\frac{1 - q^{\pair{\alp}{z}-1}}{q^{\pair{\alp}{z}} - q^{-1}} \cdot \omega(x; \sigma(z)),
$$
\textit{where $\alp$ runs over the set $\set{\alp \in \Sigma_s^+}{-\sigma(\alp) \in \Sigma^+}$ for $m = 2n$ and $\set{\alp \in \Sigma^+}{-\sigma(\alp) \in \Sigma^+}$ for $m = 2n+1$.}

\noindent
{\rm (2)} \textit{The function $G(z) \cdot \omega(x;z)$ is holomorphic and $W$-invariant, hence belongs to $\C[q^{\pm z_1}, \ldots, q^{\pm z_n}]^W$, where 
$$
G(z) = \dprod{\alp}\, \frac{1+q^{\pair{\alp}{z}}}{1-q^{\pair{\alp}{z}-1}}, 
$$
and $\alp$ runs over the set $\Sigma_s^+$ for $m = 2n$ and $\Sigma^+$ for $m = 2n+1$.}

\slit
As for the explicit formula of $\omega(x; z)$ it suffices to give for $x_\lam$ by Theorem 1 (1). 

\slit
\noindent
{\bf Theorem~3} {\it (Explicit formula) For each $\lam \in \Lam_n^+$, one has }
\begin{eqnarray*} 
\omega(x_\lam; z)
 = \frac{c_n}{G(z)} \cdot q^{\pair{\lam}{z_0
}}  \cdot Q_\lam(z; t),
\end{eqnarray*}
\textit{where $G(z)$ is given in Theorem 2,  $z_0$ is the value in $z$-variable corresponding to $s = {\bf 0}$,}
\begin{eqnarray*}
&&
c_n = \left\{\begin{array}{ll}
\dfrac{(1-q^{-2})^n}{w_m(-q^{-1})} & \textit{if } m = 2n\\
\dfrac{(1+q^{-1})(1-q^{-2})^n}{w_m(-q^{-1})} & \textit{if } m = 2n+1,
\end{array} \right. \quad w_m(t) = \prod_{i=1}^m (1-t^i),\\ 
&&
Q_\lam(z; t) = \sum_{\sigma \in W}\, \sigma\left(q^{-\pair{\lam}{z}} c(z; t) \right),\quad
c(z; t) = \prod_{\alp \in \Sigma^+}\, \frac{1 -t_\alp q^{\pair{\alp}{z}}}{1 - q^{\pair{\alp}{z}}}\\
&& 
t_\alp = \left\{ \begin{array}{ll} t_s & \textit{if } \alp \in \Sigma_s^+ \\ t_\ell & \textit{if } \alp \in \Sigma_\ell^+, \end{array}\right. \quad 
t_s = -q^{-1}, \quad t_\ell =  \left\{ \begin{array}{ll} q^{-1} & \textit{if } m = 2n\\ -q^{-2} & \textit{if } m = 2n+1. \end{array}\right.
\end{eqnarray*}

\slit
We see $Q_\lam(z; t) $ belongs to $\calR = \C[q^{\pm z_1}, \ldots, q^{\pm z_n}]^W$ by Theorem 2.  It is known that $Q_\lam(z;t) = W_\lam(t) P_\lam(z; t)$ with Hall-Littlewood polynomial $P_\lam(z;t)$ and Poincar\'{e} polynomial $W_\lam(t)$ and the set $\set{P_\lam(z; t)}{\lam \in \Lam_n^+}$ forms an orthogonal $\C$-basis for $\calR$ for each $t_\alp \in \R, \abs{t_\alp} < 1$ (cf.\ \cite[Appendix C]{HK}). As is written as above, we need the different specialization for $Q_\lam(z;t)$ according to the parity of $m$, and $G(z)$ is different also. 

In particular, we have
\begin{eqnarray} \label{at origin}   
\omega(x_0; z) = \dfrac{(1-q^{-1})^n w_n(-q^{-1})w_{m'}(-q^{-1})}{w_{m}(-q^{-1})} \cdot G(z)^{-1},\qquad m'=\left[\tfrac{m+1}{2}\right]
\end{eqnarray}
and we may modify the spherical function $\omega(x; z)$ as
\begin{eqnarray} \label{modified sph}
\Psi(x; z) = \frac{\omega(x; z)}{\omega(x_0; z)} \in \calR.
\end{eqnarray}
We define the spherical Fourier transform on the Schwartz space $\SKX$ as follows
\begin{eqnarray*}
&&
\what{\; } :  \SKX  \longrightarrow  \calR, \;  \vphi  \longmapsto  \what{\vphi}(z) = \int_{X} \vphi(x) \Psi(x;z) dx
\end{eqnarray*}
where $dx$ is a $G$-invariant measure on $X$.

\slit
\noindent
{\bf Theorem~4} 
{\it
{\rm (1)} The above spherical transform is an $\hec$-module isomorphism and $\SKX$ is a free $\hec$-module of rank $2^n$.

\noindent
{\rm (2)} All the spherical Fourier functions on $X$ are parametrized by $z \in \left(\C\big{/}\frac{2\pi\sqrt{-1}}{\log q} \right)^n \big{/} W $ through $\lam_z$, and the set $\set{\Psi(x; z + u)}{u \in \{0, \frac{\pi\sqrt{-1}}{\log q} \}^n }$ forms a $\C$-basis of spherical functions corresponding to $z$.

\noindent
{\rm (3)} (Plancherel formula) Set a measure $d\mu(z)$ on $\fra^* = \left\{\sqrt{-1}\left(\R \big{/} \frac{2\pi}{\log q} \Z \right) \right\}^n$ by
\begin{eqnarray*}
d\mu(z) = \frac{1}{2^n n!} \cdot \frac{w_n(-q^{-1}) w_{m'}(-q^{-1})}{(1+q^{-1})^{m'}} \cdot \frac{1}{\abs{c(z;t)}^2} dz,
\qquad
m'=\left[\tfrac{m+1}{2}\right]
\end{eqnarray*}
where $dz$ is the Haar measure on $\fra^*$. By an explicitly given normalization of $dx$ depending on the parity of $m$, one has
\begin{eqnarray*}
\int_{X} \vphi(x)\ol{\psi(x)} dx = \int_{\fra^*}\what{\vphi}(z) \ol{\what{\psi}(z)} d\mu(z) \quad (\vphi, \psi \in \SKX).
\end{eqnarray*}

\noindent
{\rm (4)} (Inversion formula) For any $\vphi \in \SKX$, one has%
\begin{eqnarray*}
\vphi(x)  = \int_{\fra^*} \what{\vphi}(z) \Psi(x; z) d\mu(z), \quad x \in X.
\end{eqnarray*}
}
\slit
\noindent




\vspace{2cm}
\Section{The Space $X$ of unitary hermitian matrices}


%
{\bf 1.1.}
Let $k'$ be an unramified quadratic extension of a $\mathfrak{p}$-adic field $k$ of odd residual characteristic and consider hermitian and unitary matrices with respect to $k'/k$, and denote by $a^*$ the conjugate transpose of $a \in M_{mn}(k')$.  Let $\pi$ be a prime element of $k$ and $q$ the cardinality of the residue class field $\calO_k/(\pi)$ and we normalize the absolute value on $k$ by $\abs{\pi} = q^{-1}$ and
denote by $v_\pi(\; )$ the additive value on $k$. 
We fix a unit $\epsilon \in \mathcal{O}_k^\times$ for which $k' = k(\sqrt{\epsilon})$.

We consider the unitary group
\begin{equation*}
G = G_n = \set{g \in GL_{2n+1}(k')}{g^*j_{2n+1}g = j_{2n+1}}, \qquad j_{2n+1} = \begin{pmatrix}
0 & {} & 1\\ {} & 
\iddots
& {}\\1 & {} & 0
\end{pmatrix} \in M_{2n+1},
\end{equation*}
and the space $X$ of unitary hermitian matrices in $G$
\begin{equation} \label{space X}
X = X_n  = \set{x \in G}{x^* = x, \; \Phi_{xj_{2n+1}}(t) = (t^2-1)^n(t-1)},
\end{equation}
where $\Phi_y(t)$ is the characteristic polynomial of the matrix $y$. 
It should be noted that \eqref{space X} implies $\det x=1$.
We note that $X$ is a single $G(\ol{k})$-orbit containing $1_{2n+1}$ over the algebraic closure $\ol{k}$ of $k$ (\cite[Appendix A]{HK}).
The group $G$ acts on $X$ by 
$$
g \cdot x = gxg^* 
= gxj_{2n+1}g^{-1}j_{2n+1}, \quad g \in G, \; x \in X.
$$



We fix a maximal compact subgroup $K$ of $G$ by
\begin{equation*}
  K = K_n=G\cap M_{2n+1}(\calO_{k'}),
\end{equation*}
(cf.\ \cite[\S 9]{Satake}), 
and take a  Borel subgroup $B$ of $G$, which consists of all the triangular matrices in $G$ and is given by
\begin{equation}
\label{eq:B}
B = \set{
  \begin{pmatrix}
    A&&\\
    &u&\\
    &&j_nA^{*-1}j_n
  \end{pmatrix}
  \begin{pmatrix}
    1_n&\beta&C\\
    &1&-\beta^*j_n\\
    &&1_n
  \end{pmatrix}}{
  \begin{aligned}
    &A\in B_n,\beta\in(k')^n,u\in\mathcal{O}^1_{k'},C\in M_n(k')\\
    &\beta\beta^*+Cj_n+j_nC^*=0_n
  \end{aligned}
},
\end{equation}
where $B_n$ is the set of all the upper triangular matrices in $GL_n(k')$,
 $\calO_{k'}^1 = \set{u \in \calO_{k''}^\times}{N(u) = 1}$.
Here and hereafter empty entries in matrices should be understood as $0$ and $N$ as the norm map $N_{k'/k}$.
The group $G$ satisfies the Iwasawa decomposition $G=BK=KB$.

In this section, we give the $K$-orbit decomposition and the $G$-orbit decomposition of the space $X$.

\bigskip
\begin{thm} \label{thm: Cartan}
The 
$K$-orbit decomposition of $X_n$ is given as follows:
\begin{equation} \label{K-orbits}
X_n = \bigsqcup_{\lambda \in \Lambda_n^+} K \cdot x_\lambda,
\end{equation}
where 
\begin{align*}
&
\Lambda_n^+ = \set{\lambda \in \mathbb{Z}^n}{\lambda_1 \geq \cdots \geq \lambda_n \geq 0},\\
&
x_\lambda 
= \Diag(\pi^{\lambda_1}, \ldots, \pi^{\lambda_n}, 1, \pi^{-\lambda_n}, \cdots, \pi^{-\lambda_1}).
\end{align*}
\end{thm}

\bigskip
We recall the case of unramified hermitian matrices. The group $GL_m(k')$ acts on the set $\calH_m(k') = \set{x \in GL_m(k')}{x^* = x}$ by $g \cdot x  = gxg^*$, and it is known (cf. \cite{Jac})
\begin{eqnarray} \label{herm Cartan}
\calH_m(k') &=& \bigsqcup_{\lambda \in \Lambda_m} GL_m(\calO_{k'}) \cdot \pi^\mu  
\end{eqnarray}
where
\begin{eqnarray*}
\Lam_m = \set{\mu \in \Z^m}{\mu_1 \geq \cdots \geq \mu_m}, \quad \pi^\mu = \Diag(\pi^{\mu_1},\ldots, \pi^{\mu_m}).
\end{eqnarray*}
Hence, we see that $K \cdot x_\lam \cap K \cdot x_\mu = \emptyset$ if $\lam \ne \mu$ in $\Lam_n^+$. Moreover, since $N(\calO_{k'}^\times ) = \calO_k^\times$, we see that any diagonal $x \in X$ is reduced to some $x_\lam, \; \lam \in \Lam_n^+$ by the action of $K$.

\bigskip
\noindent
{\bf 1.2.}
In this subsection we prove Theorem \ref{thm: Cartan} for $n=1$. It is useful to write down the explicit form of $K_1$, which is easily checked.
\begin{lem} \label{explicit K1}
\begin{eqnarray*}
&&
K_1 = K_{1,1} \sqcup K_{1,2},\\
&&
K_{1,1}:= \set{g \in BjB \cap K}{g_{31} \in \calO_{k'}^\times} = \set{g \in K}{g_{31} \in \calO_{k'}^\times}\\[2mm]
&& \quad = \set{
\begin{pmatrix}\alp &&\\&u&\\&& \alp^{* -1} \end{pmatrix}
\begin{pmatrix}1  &-d^*&f\\&1&d\\&&1 \end{pmatrix}
\begin{pmatrix}  &&1\\&1&-b^*\\ 1&b & c \end{pmatrix}}
{\begin{array}{l}
  \alp \in \calO_{k'}^\times, \; u \in \calO_{k'}^1\\
b, d \in \calO_{k'}, \; c_0, f_0 \in \calO_k\\
c = -\frac{N(b)}{2} + c_0\sqrt{\ve}, \; f = -\frac{N(d)}{2} + f_0\sqrt{\ve}
\end{array}},\nonumber \\[2mm]
&&
K_{1,2}:= \set{g \in K}{g_{31} \in \frp}\\[2mm]
&&\quad
=\set{
\begin{pmatrix}\alp &&\\& u &\\&&\alp^{* -1} \end{pmatrix}
\begin{pmatrix}1  &&\\ b & 1 &  \\ c& -b^* & 1 \end{pmatrix}
\begin{pmatrix}  1& d & f\\ &1&-d^*\\  &  & 1 \end{pmatrix}}
{\begin{array}{l}
\alp \in \calO_{k'}^\times, \; u \in \calO_{k'}^1\\
b \in \pi\calO_{k'}, \; c_0 \in \pi\calO_k, \; d \in \calO_{k'}, \; f_0 \in \calO_k\\
c = -\frac{N(b)}{2} + c_0\sqrt{\ve}, \; f = -\frac{N(d)}{2} + f_0\sqrt{\ve}
\end{array} }.
\end{eqnarray*}
\end{lem}

\bigskip
We take an element $x \in X_1$, write it as 
\begin{eqnarray} \label{any x}
x = \begin{pmatrix}
  a & b & c \\
  b^* & d & f \\
  c^* & f^* & g
\end{pmatrix}, \qquad a, d, g\in k, \quad b, c, f\in k',
\end{eqnarray}
and show that the orbit $K\cdot x$ contains an element $\Diag(\pi^\ell, 1, \pi^{-\ell})$ for some $\ell \geq 0$. Since 
$x\in G$ is equivalent to $xj_3xj_3=1_3$ and $\Phi_{xj_3}(t)=(t^2-1)(t-1)$,
we obtain the following equations 
\begin{subequations}
\label{eq:fund_eqs}
\begin{align}
\label{eq:fund_eqs1}
    ag+bf+c^2&=1, \\
\label{eq:fund_eqs2}
    af^*+b(c+d)&=0,\\
\label{eq:fund_eqs3}
    a(c+c^*)+bb^*&=0,\\
\label{eq:fund_eqs4}
    b^*g+(c+d)f&=0,\\
\label{eq:fund_eqs5}
    bf+b^*f^*+d^2&=1,\\
\label{eq:fund_eqs6}
    (c+c^*)g+ff^*&=0,
\end{align}
\end{subequations}
and
\begin{eqnarray} \label{eq:char_poly}
\lefteqn{(t^2-1)(t-1)}\\
&=&(t-c)(t-c^*)(t-d)-(t-c)b^*f^*
-(t-c^*)bf-(t-d)ag-aff^*-bb^*g. \nonumber
\end{eqnarray}

\begin{lem}
\label{lem:orbit1}
If the element $x$ of (\ref{any x}) satisfies the following condition {\rm (i)} or {\rm (ii)}
\begin{enumerate}\def\labelenumi{\rm(\roman{enumi})}
\item $a \ne 0$ and $v_\pi(a)\leq v_\pi(b)$,  or $g \ne 0$ and $v_\pi(g)\leq v_\pi(f)$,
\item $ag=0$,
\end{enumerate}
then $K \cdot x$ contains $\Diag(\pi^\ell, 1, \pi^{-\ell})$ for some $\ell \geq 0$.
\end{lem}

\proof 
As for the condition (i), it suffices to consider the case $a \ne 0$ and $v_\pi(a)\leq v_\pi(b)$, by the action of $j_3$. 
Considering the action of
$$
\begin{pmatrix}
  1 &&\\
\lambda &1& \\
\mu&-\lambda^*&1
\end{pmatrix}\in K, \quad \lambda,\mu\in\mathcal{O}_{k'}\; \mbox{such that}\; a\lambda+b^*=0, \; \lambda\lambda^*+\mu+\mu^*=0,
$$
we can assume
\begin{equation*}
  x=
  \begin{pmatrix}
    a & 0 & c \\
    0 & d & f \\
    c^* & f^* & g
  \end{pmatrix}.
\end{equation*}
Then \eqref{eq:fund_eqs} and
\eqref{eq:char_poly} imply
\begin{equation*}
  ag+c^2=1,\qquad
  f=0,\qquad 
  c+c^*=0,\qquad 
  d=1,
\end{equation*}
and therefore we can write as, after the action of $j_3$ if necessary,  
\begin{equation*}
  x=
  \begin{pmatrix}
    a & 0 & c_1\sqrt{\epsilon} \\
    0 & 1 & 0 \\
    -c_1\sqrt{\epsilon} & 0 & g
  \end{pmatrix}, \quad a,g,c_1\in k, \; ag+c_1^2\epsilon=1, v_\pi(a) \geq v_\pi(g).
\end{equation*}
Since $\sqrt{\eps} \notin k$, we see $v_\pi(ag) \leq 0$,  and by a suitable element of type $\Diag(\alp, 1, \alp^{* -1}) \in K$, we may assume $g = \pi^{-\ell}$ with $\ell \geq 0$. 
Since $v_\pi(c_1) \geq -\ell$, by the action of
\begin{equation*}
  \begin{pmatrix}
    1 & 0 & -c_1\pi^{\ell}\sqrt{\epsilon} \\
    0 & 1 & 0\\
    0 & 0 & 1
  \end{pmatrix}\in K,
\end{equation*}
we see that $K \cdot x$ contains $\Diag(\pi^\ell, 1, \pi^{-\ell})$.

As for the condition (ii), it suffices to consider the case $a=0$ by the action of $j_3$. From \eqref{eq:fund_eqs}, we have $b=0$ and hence 
$c=c^*=\pm1$ and $d=\pm1$.
Then $c=1$, $d=-1$ by \eqref{eq:char_poly}. Thus we have 
\begin{equation*}
  x=
  \begin{pmatrix}
    0 & 0 & 1 \\
    0 & -1 & f \\
    1 & f^* & g
  \end{pmatrix},
  \qquad 2g+ff^*=0.
\end{equation*}
Since
\begin{equation*}
  \Diag(u^{* -1},1,u)\cdot x=
  \begin{pmatrix}
    0 & 0 & 1 \\
    0 & -1 & u^*f \\
    1 & uf^* & uu^*g
  \end{pmatrix},    
\end{equation*}
we see that
\begin{equation*}
  K\cdot x\ni
  \begin{pmatrix}
    0 & 0 & 1 \\
    0 & -1 & h \\
    1 & h & -\frac{1}{2}h^2
  \end{pmatrix},
\qquad h=0\text{ or }h=\pi^{\ell}\quad(\ell\in\mathbb{Z}).
\end{equation*}
If $h = \pi^\ell$ with $\ell\leq 0$, it reduces to the case (i) and we have done. 
If $h = \pi^\ell$ with $\ell\geq 1$, we see
\begin{equation*}
K \cdot x \ni   \begin{pmatrix}
    1 & 0 & 0 \\
    -\frac{h}{2} & 1 & 0 \\
    -\frac{h^2}{8} & \frac{h}{2} &1
  \end{pmatrix} \cdot x =   
\begin{pmatrix}
    0 & 0 & 1 \\
    0 & -1 & 0 \\
    1 & 0 & 0
  \end{pmatrix}.
\end{equation*}
Since we have 
\begin{equation*}
K \cdot x \ni 
  \begin{pmatrix}
    1 & -1 & -1/2 \\
     & 1 & 1 \\
     & & 1
  \end{pmatrix}
  \begin{pmatrix}
 & & 1\\
 &1 & 1\\
1 & -1 &-1/2
  \end{pmatrix} \cdot \begin{pmatrix}
    0 & 0 & 1 \\
    0 & -1 & 0 \\
    1 & 0 & 0
  \end{pmatrix} = \Diag(-\frac12, 1, -2),
\end{equation*}
and $2 \in N(\calO_{k'}^\times)$,
we see $K \cdot x  \ni 1_3$, 
which completes the proof of Lemma~\ref{lem:orbit1}.
\qed

\begin{lem}
\label{lem:orbit2}
Assume the element $x$ of (\ref{any x}) satisfies the following condition
\begin{equation}
  \label{eq:cond_abgf}
  ag\neq 0, \quad v_\pi(a)>v_\pi(b), \quad \mbox{and}\quad v_\pi(g)>v_\pi(f).
\end{equation}
Then the orbit $K\cdot x$ contains the following matrix 
\begin{eqnarray}
\label{eq:diag4}
&&
j_3+
\begin{pmatrix}
    \pi^{2m}s^{-1}&\pi^m & r\pi^{m+\ell} \\
    \pi^m & s &\pi^\ell rs \\
    r^*\pi^{m+\ell}& \pi^\ell r^*s & \pi^{2\ell} rr^*s
  \end{pmatrix},\quad 
\begin{array}{l}
   m \geq \ell > 0\\
   r \in \mathcal{O}_{k'}^\times, \; s \in \calO_k^\times \\
  \pi^{m + \ell}(r+r^*) + s + 2 = 0,
\end{array} \nonumber \\
&=&
j_3 + \frac{1}{s}{\bf v}{\bf v^*}, \quad {\bf v} = \begin{pmatrix}\pi^m \\ s\\ \pi^\ell r^*s \end{pmatrix}.
\end{eqnarray} \nonumber
\end{lem}

\proof
  Without loss of generality, we assume $v_\pi(a)\geq v_\pi(g)$. From \eqref{eq:fund_eqs2} and
\eqref{eq:fund_eqs4}, we have
$aff^*=bb^*g$, which implies
\begin{equation}
  \label{eq:cond_agbf}
  v_\pi(a)-v_\pi(g)=2(v_\pi(b)-v_\pi(f))\geq0.
\end{equation}
We show that $v_\pi(c)\geq0$. If $v_\pi(c)<0$, then
by \eqref{eq:fund_eqs1} and \eqref{eq:cond_abgf} we have
\begin{equation*}
  v_\pi(bf)=2v_\pi(c)\leq -2,
\end{equation*}
and in particular, $v_\pi(f)<0$ by \eqref{eq:cond_agbf}.
Again in \eqref{eq:fund_eqs1}, we obtain
\begin{equation*}
  bf^{-1}+(cf^{-1})^2=f^{-2}-(af^{-1})(gf^{-1})\in(\pi).
\end{equation*}
Then, since $bf^{-1}\in\mathcal{O}_{k'}$, we have $cf^{-1} \in \mathcal{O}_{k'}$. Hence 
$(\pi)\ni(cf^{-1}+c^*f^{-1})gf^{*-1}$, which is equal to $-1$ by \eqref{eq:fund_eqs6}, and we arrive at the contradiction.
Thus  $v_\pi(c) \geq 0$.
Then, by \eqref{eq:cond_abgf} and \eqref{eq:fund_eqs6}, we see $v_\pi(f)<v_\pi(g)\leq2v_\pi(f)$, which implies $f\in(\pi)$ and hence $a, b, g\in(\pi)$.
Thus by \eqref{eq:fund_eqs1} and \eqref{eq:fund_eqs5}, we obtain
\begin{equation*}
  c\equiv \pm 1,\quad d\equiv \pm 1\pmod{(\pi)}
\end{equation*}
and by
\eqref{eq:char_poly}, 
\begin{equation}
\label{eq:cond_cd}
  c\equiv 1,\quad d\equiv -1\pmod{(\pi)}.
\end{equation}
By
\eqref{eq:fund_eqs3} and \eqref{eq:fund_eqs6}, we obtain
\begin{equation*}
  v_\pi(a)=2b_\pi(b)>0,\quad
  v_\pi(g)=2b_\pi(f)>0.
\end{equation*}
Hence we can assume 
  \begin{equation}
\label{eq:diag2}
  x=\begin{pmatrix}
  \pi^{2m}&\pi^m u&c\\
  \pi^m u^*&d&\pi^\ell v\\
  c^*& \pi^\ell v^*& \pi^{2\ell}w
\end{pmatrix},
\qquad
\begin{aligned}
  &m\geq\ell>0,\\
  &u,v,c\in\mathcal{O}_{k'}^\times, \quad
  w,d\in\mathcal{O}_k^\times,\\
  & c\equiv c^*\equiv 1,d\equiv -1\pmod{(\pi)},
\end{aligned}
\end{equation}
and the set of equations \eqref{eq:fund_eqs} becomes
\begin{subequations}
\label{eq:nfund_eqs}
\begin{align}
\label{eq:nfund_eqs1}
  &\pi^{2(m+\ell)}w+\pi^{m+\ell}uv+c^2=1,\\
\label{eq:nfund_eqs2}
  &c+d=-\pi^{m+\ell}u^{-1}v^*,\\
\label{eq:nfund_eqs3}
  &c+c^*=-uu^*,\\
\label{eq:nfund_eqs4}
  &c+d=-\pi^{m+\ell}u^*v^{-1}w,\\
\label{eq:nfund_eqs5}
  &\pi^{m+\ell}(uv+u^*v^*)+d^2=1,\\
\label{eq:nfund_eqs6}
  &c+c^*=-vv^*w^{-1}.
\end{align}
\end{subequations}
By \eqref{eq:nfund_eqs3} and \eqref{eq:nfund_eqs6}, we see $w = (uu^*)^{-1}vv^*$. By \eqref{eq:char_poly} we see
\begin{eqnarray} \label{eq:c-d}
c+c^*+d=1. 
\end{eqnarray}
Then, we see $c = 1+\pi^{m+\ell}u^{* -1}v$ by \eqref{eq:nfund_eqs2}, and $d = 1+uu^*$ by \eqref{eq:nfund_eqs3}.
Setting $r = u^{-1 *}v \in \calO_{k'}^\times$ and $s = uu^* \in \calO_k^\times$, we have  $\pi^{m+\ell}(r+r^*) + s + 2 = 0$ by \eqref{eq:nfund_eqs1} and 
\begin{eqnarray}
    \Diag(u^{-1},1,u^*)\cdot x
&=&
    \begin{pmatrix}
      \pi^{2m}(uu^*)^{-1} & \pi^m & c \\
      \pi^m & d &\pi^\ell uv\\
      c^* & \pi^\ell u^*v^*& \pi^{2\ell}uu^*w
    \end{pmatrix}   \nonumber
    \\ 
    &=&
j_3 +     \begin{pmatrix}
      \pi^{2m}s^{-1}&\pi^m & \pi^{m+\ell}r \\
      \pi^m & s &\pi^\ell rs \\
      \pi^{m+\ell}r^*& \pi^\ell r^*s & \pi^{2\ell} rr^*s
    \end{pmatrix} \nonumber\\
&=&
j_3 + \frac{1}{s} {\bf v}{\bf v}, \quad {\bf v } = \begin{pmatrix} \pi^m \\s \\ \pi^\ell r^*s \end{pmatrix}, 
\end{eqnarray}
which completes the proof.
\qed

\bigskip
It is easy to check the matrix in \eqref{eq:diag4} is an element of $X \cap K$, i.e. it is integral and the entries satisfies the condition \eqref{eq:fund_eqs}.
The following statement finally shows Theorem~\ref{thm: Cartan} in the case $n=1$.

\begin{lem}
\label{lem:orbit3}
  The $K$-orbit containing the matrix $j_3 + \frac{1}{s}{\bf v}{\bf v}^*$ in \eqref{eq:diag4} contains $1_3$. 
\end{lem}

\begin{proof}
For any $k \in K$, we have
 \begin{equation*}
k \cdot (j_3 + \frac{1}{s}{\bf v}{\bf v}^*) = j_3 +  \frac{1}{s}(k{\bf v})(k{\bf v})^*.
\end{equation*}
If the second entry $(k{\bf v})_2$ of $k{\bf v}$ becomes $0$, we see 
\begin{eqnarray*}
k \cdot (j_3 + \frac{1}{s}{\bf v}{\bf v}^*) = \begin{pmatrix} * & 0 & *\\ 0 & 1 & 0\\ * & 0 & * \end{pmatrix},
\end{eqnarray*}
which satisfies 
the assumption of Lemma~\ref{lem:orbit1}, 
hence it is diagonalized to $1_3$.

In the following we will show that there exists $k \in K$ for which $(k{\bf v})_2 = 0$ and of the form 
\begin{equation}
  k=
  \begin{pmatrix}
    1 & -b^* & c \\
     & 1 & b\\
     &  & 1
  \end{pmatrix}
  \begin{pmatrix}
     &  &1 \\
     & 1 &-d^*\\
    1 & d& f
  \end{pmatrix}
=
\begin{pmatrix}
  * & * & * \\
  b & 1+bd & bf-d^*\\
 * & * & *
\end{pmatrix}. \label{eq: the form}
\end{equation}
For the matrix of the form \eqref{eq: the form}, $k \in K$ if and only if $b,c,d,f\in\mathcal{O}_{k'}$ and $bb^*+c+c^*=dd^*+f+f^*=0$.
For simplicity, we put $y=s^{-1}\pi^{m-\ell}(r+r^*)\in\mathcal{O}_k^\times$.
We define a unit $\gamma\in\mathcal{O}_{k}^\times$ by
\begin{equation}
  \label{eq:def_gamma}
  \gamma=\frac{-\rho rr^*}{2\rho+y\rho^2-\pi^{2\ell}},
\end{equation}
where
\begin{eqnarray*}
\rho = \left\{\begin{array}{ll}
1 & \mbox{if } v_\pi(y) > 0\\
-y^{-1} & \mbox{if } v_\pi(y) = 0
\end{array} \right. , \quad
2\rho + y\rho^2-\pi^{2\ell} \in \calO_k^\times.
\end{eqnarray*}
Since
$\rho\gamma\in\mathcal{O}_{k}^\times$, we can take $b\in\mathcal{O}_{k'}^\times$
such as 
\begin{equation}
  \label{eq:def_b}
  bb^*=\rho\gamma.
\end{equation}
Further we put
\begin{align*}
  c&=-bb^*/2,\\
  d&=b^{-1}(-1+\pi^{\ell}(b^*r)^{-1}\gamma),\\
  f&=-\pi^{m-\ell}r^{*-1}s^{-1}-\pi^{-\ell}(br^*)^{-1}(1+bd)+b^{-1}d^*.
\end{align*}
We see that
$b,c,d, f\in\mathcal{O}_{k'}$, where we use the relation
\begin{eqnarray} \label{eq:1+d}
1+bd=\pi^{\ell}(b^*r)^{-1}\gamma.
\end{eqnarray}
It is trivial that $bb^*+c+c^*=0$. We show that $dd^*+f+f^*=0$.
We have
\begin{equation}
\label{eq:rrbbdd}
  \begin{split}
    rr^*bb^*dd^*
    &=rr^*(-1+(b^*r)^{-1}\gamma\pi^{\ell})(-1+(br^*)^{-1}\gamma\pi^{\ell})
    \\
    &=rr^*-(b^{*-1}r^*+b^{-1}r)\gamma\pi^{\ell}+(bb^*)^{-1}\gamma^2\pi^{2\ell}.
  \end{split}
\end{equation}
On the other hand, using \eqref{eq:1+d}, we obtain
\begin{equation}
\label{eq:rrbbff}
  \begin{split}
    rr^*bb^*(f+f^*)&=-s^{-1}\pi^{m-\ell}bb^*r-b^*(1+bd)\pi^{-\ell}r+b^*d^*rr^*\\
    &\qquad-s^{-1}\pi^{m-\ell}bb^*r^*-b(1+b^*d^*)\pi^{-\ell}r^*+bdrr^*\\
    &=-s^{-1}\pi^{m-\ell}bb^*(r+r^*)+(bd+b^*d^*)rr^*-2\gamma\\
    &=-ybb^*+((b^*r)^{-1}\gamma\pi^{\ell}+(br^*)^{-1}\gamma\pi^{\ell}-2)rr^*-2\gamma\\
    &=-ybb^*+(b^{*-1}r^*+b^{-1}r)\gamma\pi^{\ell}-2rr^*-2\gamma.
  \end{split}
\end{equation}
Combining 
\eqref{eq:def_gamma}, \eqref{eq:def_b},
\eqref{eq:rrbbdd} and \eqref{eq:rrbbff}, we arrive at
\begin{equation*}
  \begin{split}
    rr^*bb^*(dd^*+f+f^*)
    &=(bb^*)^{-1}\gamma^2\pi^{2\ell}-ybb^*-rr^*-2\gamma
    \\
    &=(\rho\gamma)^{-1}\gamma^2\pi^{2\ell}-y\rho\gamma-rr^*-2\gamma
    \\
    &=-\rho^{-1}\gamma(y\rho^2+rr^*\rho\gamma^{-1}+2\rho-\pi^{2\ell})
    \\
    &=0.
  \end{split}
\end{equation*}
Thus we have shown that $dd^*+f+f^*=0$, and hence  $k \in K$.

By definition of $f$,
we obtain
\begin{equation*}
  bf-d^*=-\pi^{m-\ell}bs^{-1}r^{*-1}-\pi^{-\ell}(1+bd)r^{*-1},
\end{equation*}
which implies
\begin{equation}
  \begin{split}
    (k{\bf v})_2&=    
    \begin{pmatrix}
      b & 1+bd & bf-d^*
    \end{pmatrix}
    \begin{pmatrix}
      \pi^m \\ s \\ s\pi^{\ell}r^*
    \end{pmatrix}\\
    &=b\pi^m+(1+bd)s
    -b\pi^m-(1+bd)s
    \\
    &=0.
  \end{split}
\end{equation}
\end{proof}

\bigskip
\noindent
{\bf 1.3.}
In the following we will show Theorem~\ref{thm: Cartan} in the case $n\geq 2$.
Our strategy is the same as in \cite{HK}. We often use elements in $B \cap K$, $j_{2n+1}Bj_{2n+1} \cap K$ or 
$$
\set{\begin{pmatrix} A &&\\&u&\\&&j_nA^{* -1}j_n \end{pmatrix}}{A \in GL_n(\calO_{k'}), u \in \calO_{k'}^1} (\subset K).
$$
For $a = (a_{ij}) \in M_{2n+1}(k')$, we set
\begin{eqnarray*}
-\ell(a) = \min\set{v_\pi(a_{ij})}{1 \leq i, j \leq 2n+1},
\end{eqnarray*}
and say an entry of $a$ to be {\it minimal} if its $v_\pi$-value is $-\ell(a)$.
For $g \in G$, we see $\ell(g) \geq 0$, since $v_\pi(\det(g)) = 0$.

\begin{lem}\label{lem:orbit4}
  Let $n\geq2$ and assume that $x\in X_n$ has
a minimal entry in the diagonal except the $(n+1,n+1)$-entry.
Then $K \cdot x$ contains a hermitian matrix of the type
$$
\left( \begin{array}{c|c|c}
\pi^\ell & 0 & 0\\ 
\hline
0 & y & 0\\  
\hline
0 & 0 & \pi^{-\ell}
\end{array} \right), \qquad y \in X_{n-1} \cap M_{2n-1}(\pi^{-\ell}\mathcal{O}_{k'}), 
$$
where $\ell = \ell(x)$. 
\end{lem}

\begin{proof}
By the action of $W$, we may assume the $(2n+1, 2n+1)$-entry is minimal.
Then, it is easy to see that $K\cdot x$ contains an element $x'$ as follows
$$
x' = 
\left( \begin{array}{c|ccc|c}
a & a_2 & \cdots & a_{2n} & b\\
\hline
a_2^* &  &  &  & 0\\
\vdots &  & y &  & \vdots\\
a_{2n}^* &  &  &  & 0\\
\hline
b^* & 0 & \cdots & 0 & \pi^{-\ell} \end{array} \right). 
$$
By the relation $x'j_{2n+1}x'j_{2n+1} = 1_{2n+1}$ we see $b + b^* = 0$, hence we may assume $b = 0$ by the action of $K \cap B$. Then,
still by the same relation, we have $a =\pi^\ell$ and $a_i = 0, \; i \geq 2$. At that time $y=y^* \in M_{2n-1}(\pi^{-\ell}\calO_{k'})$ and
\begin{equation*}
  (t^2-1)^n(t-1)=\Phi_{xj_{2n+1}}(t)=(t^2-1)\Phi_{yj_{2n-1}},
\end{equation*}
which completes the proof.
\end{proof}

\bigskip
The following two lemmas can be shown in the same way as in \cite[Lemma 1.4 and 1.5]{HK}, so we omit the proof.
\begin{lem}   \label{lem:orbit5}
Let $n \geq 2$ and assume that $x \in X_n$ has a minimal entry outside of the diagonal, the anti-diagonal, the $(n+1)$th row, and the $(n+1)$th column. Then $K \cdot x$ contains a hermitian matrix of type
$$
\left( \begin{array}{c|c|c}
\begin{array}{cc}
\pi^{\ell}&0\\ 0 &\pi^{\ell}  
\end{array}
 & 0 & 0\\
\hline 
0 & y & 0\\
\hline  
0 & 0 & 
\begin{array}{cc}
\pi^{-\ell}&0\\ 0 &\pi^{-\ell}  
\end{array}
\end{array}\right), \qquad y \in X_{n-2} \cap M_{2n-3}(\pi^{-\ell}\mathcal{O}_{k'}), 
$$
where $\ell = \ell(x)$. 
\end{lem}

\begin{lem}   \label{lem:orbit6}
Let $x \in X_n$ with $n \geq 2$. Assume that any minimal entry of $x \in X_n$ stands in the anti-diagonal except the $(n+1,n+1)$-entry but 
some entries of the anti-diagonal except the $(n+1,n+1)$-entry are not minimal. 
Then $K \cdot x$ contains a hermitian matrix of the same type as in Lemma~\ref{lem:orbit5}.
\end{lem}

\begin{lem}
\label{lem:orbit7}
  Let $x \in X_n$ with $n \geq 2$. Assume that any minimal entry of $x \in X_n$ stands in the anti-diagonal or, the $(n+1)$th row or, the $(n+1)$th column
and that no other entries are minimal. 
Then $K \cdot x$ contains a hermitian matrix of the type 
\begin{equation*}
\left( \begin{array}{l|c|l}
  1_2 & 0 & 0\\
  \hline
  0 & y & 0\\
  \hline
  0 & 0 & 1_2
\end{array} \right) \in K, \quad y \in X_{n-2} \cap M_{2n-3}(\mathcal{O}_{k'}).
\end{equation*}

\end{lem}

\begin{proof}
  By the action of $W$, the minimal element, say $\alpha$,
  among $(j,n+1)$-entries for $1\leq j\leq 2n+1$, $j\neq n+1$ can be moved to
  $(1,n+1)$. Then by appropriate elements of $K$, all the entries
  $(j,n+1)$ for $2\leq j\leq 2n$, $j\neq n+1$ are eliminated.
  Thus we can assume that
  \begin{equation*}
    x=\left(
    \begin{array}{cccc|c|cccc}
      &&&& \alpha &&&& \xi_1 \\
      &*&&&    0   &&*&\iddots& \\
      &&&& \vdots &&\iddots&*& \\
      &&&& 0 & \xi_n &&& \\
      \hline
      \alpha^* & 0 &\cdots  &0 & u&0 &\cdots& 0 &\kappa^*\\
      \hline
      &&& \xi_n^*& 0 &&&& \\
      &*&\iddots&& \vdots &&&& \\
      &\iddots&*&& 0 &&&*& \\
       \xi_1^*&&&& \kappa &&&&\\
    \end{array}
  \right),
  \end{equation*}
  where $v_\pi(\alpha)\leq v_\pi(\kappa)$.
  Let $\ell=\ell(x)$. By $xj_{2n+1}xj_{2n+1}=1_{2n+1}$, we have
the set of equations 
\begin{gather}
\label{eq:xjxj1}
  \alpha\kappa^*+\xi_1^2\equiv1,\qquad \alpha\alpha^*\equiv0,\qquad \kappa\kappa^*\equiv0 \pmod{(\pi^{-2\ell+1})},\\
\label{eq:xjxj2}
u^2+\alpha\kappa^*+\alpha^*\kappa=1,
\end{gather}
which arise respectively from $(1,1)$, $(1,2n+1)$, $(2n+1,1)$ and $(n+1,n+1)$-entries.
Since
$\alpha\alpha^*, \kappa\kappa^* \in (\pi^{-2\ell+1})$, we see 
$\alpha, \kappa \in (\pi^{-\ell+1})$ and they are not minimal elements.

Look at the identity \eqref{eq:xjxj2}.
If $v_\pi(u) < 0$, then $2v_\pi(u) = v_\pi(\alpha\kappa^*+\alpha^*\kappa) \geq -2\ell + 2$ and $u$ is not minimal. If $v_\pi(u) \geq 0$, there must be a minimal entry among $\{\xi_i\}$ (otherwise $x$ is degenerate).
Hence in any cases, by Lemma \ref{lem:orbit6}, we may assume all the $\xi_i$ are minimal, i.e., 
$0 \geq v_\pi(\xi_i) = -\ell$, $1 \leq i \leq n$.

Then by \eqref{eq:xjxj2} and (1,1)-entry of $xj_{2n+1}xj_{2n+1}$, we see $\xi_1^2- u^2 \in (\pi^{-2\ell +2})$, hence $\xi_1-u$ or $\xi_1 + u$ is contained in $(\pi^{-\ell+1})$, and $v_\pi(u) = v_\pi(\xi_1) = -\ell$. By the above argument, we have $v_\pi(u)= v_\pi(\xi_i) = 0, 1 \leq i \leq n$.
Then, since $u^2 \equiv \xi_i^2 \equiv 1 \pmod{(\pi)}$, we see $u \equiv \pm 1, \; \xi_i \equiv \xi_i^* \equiv \pm 1 \pmod{(\pi)}$. 

By the condition of the characteristic polynomial, we have
$$
(t^2-1)^n(t-1) \equiv \prod_{i=1}^n (t - \xi_i)^2(t-u) \pmod{(\pi)},
$$
and we see $\xi_i \not\equiv \xi_j$ for some $i, j$. Then by the same argument as in \cite[Lemma 1.6 (ii)]{HK}, we finally obtain the assertion.
\end{proof}


\bigskip
Now Lemmas \ref{lem:orbit1} to \ref{lem:orbit7}
complete the proof of
Theorem \ref{thm: Cartan}.

\begin{rem}
If $k$ has even residual characteristic, then
there are some $K$-orbits without any diagonal matrix.
In fact,
the following matrix is contained in $X_n$ and can not be diagonalized, for $n = r+s, \; s>0$:
\begin{equation*}
  \begin{pmatrix}
  1_r&&&& \\
  & && -j_{\frac{s}{2}}& \\
  && j_{s+1} && \\
&-j_{\frac{s}{2}} &&&\\
&&&&1_r
\end{pmatrix}  \quad (\mbox{if } 2 \mid s), \quad
  \begin{pmatrix}
  1_r&&&& \\
  & && j_{\frac{s+1}{2}}& \\
  && -j_{s} && \\
&j_{\frac{s+1}{2}} &&&\\
&&&&1_r
\end{pmatrix}  \quad (\mbox{if } 2 \not{\mid}\; s).
\end{equation*}
\end{rem}


\proof
  We show this assertion by contradiction in a slightly general situation.
Set 
$$
x = x^* \in K \; \mbox{such that} \; x \equiv \begin{pmatrix} 1_r & & \\ & j_{2s+1}&\\&&1_r \end{pmatrix} \pmod{(\pi)}, \quad r+s=n, \; s > 0
$$
and assume there exists an element $g \in K$ which diagonalizes $x$ and write down
  \begin{equation*}
    \begin{split}
      &g=
      \begin{pmatrix}
        g_1&g_2&g_3\\
        k_1&k_2&k_3\\
        h_1&h_2&h_3
      \end{pmatrix},
            \quad \begin{array}{l} 
      g_1,h_1,g_3,h_3\in M_{r+s, r}(\mathcal{O}_{k'}), \; k_1,k_3\in M_{1, r}(\mathcal{O}_{k'})\\
k_2\in M_{1, 2s+1}(\mathcal{O}_{k'}),\;     g_2,h_2\in M_{r+s, 2s+1}(\mathcal{O}_{k'}). 
\end{array}
    \end{split}
  \end{equation*}
Then, since $g \cdot x$ is diagonal in $K$, we can assume $g\cdot x = 1_{2n+1}$. Since $g^* = j_{2n+1}g^{-1}j_{2n+1}$, we have
\begin{eqnarray*}
g \begin{pmatrix}&&j_r\\ &1_{2s+1} &\\ j_r \end{pmatrix} \equiv \begin{pmatrix}&&j_{r+s}\\ &1&\\ j_{r+s} \end{pmatrix} g \pmod{(\pi)},
\end{eqnarray*}
which implies
\begin{eqnarray*}
&&
g \equiv \begin{pmatrix} g_1 & g_2 & g_3 \\ k_1 & k_2 & k_1j_r \\j_{r+s}g_3j_r & j_{r+s}g_2 & j_{r+s}g_1j_r \end{pmatrix}   \\
&\equiv&
\begin{pmatrix}  1_{r+s+1}& \\ &j_{r+s}\end{pmatrix}
\begin{pmatrix} g_1 & g_2 & g_3j_r \\ k_1 & k_2 & k_1 \\g_3j_r & g_2 & g_1 \end{pmatrix} 
\begin{pmatrix}  1_{r+2s+1}& \\ &j_{r}\end{pmatrix} \pmod{(\pi)}.
\end{eqnarray*}
\newcommand{\rank}{{\rm rank}}
Here we have
\begin{eqnarray*}
&&
\rank\,  g\! \pmod{(\pi)}
= \rank
\begin{pmatrix} g_1+g_3j_r & 0 & g_1+g_3j_r \\ k_1 & k_2 & k_1 \\g_3j_r & g_2 & g_1 \end{pmatrix}\! \pmod{(\pi)} \\
&=&
\rank \begin{pmatrix} 0 & 0 & g_1+g_3j_r \\ 0 & k_2 & k_1 \\g_1+g_3j_r & g_2 & g_1 \end{pmatrix}\! \pmod{(\pi)} 
\\
&\leq &2r+s + 1 < n,
\end{eqnarray*}
which is a contradiction to $g \in K$.


\bigskip
\noindent
{\bf 1.4.} In this subsection, we give the $G$-orbit decomposition of $X_n$. First we recall the case of unramified hermitian matrices. It is known that there are precisely two $GL_m(k')$-orbits in $\calH_m(k')$ for $m \geq 1$:
\begin{eqnarray} \label{eq:herm-G}
\calH_m(k') &=& GL_m(k')\cdot 1_m \sqcup GL_m(k') \cdot \pi^{(1,0,\ldots,0)},\nonumber\\ 
&=& \left(\displaystyle{\sqcup_{\shita{\mu \in \Lam_m}{\abs{\mu}\, is\, even}}} GL_m(\calO_{k'})\cdot \pi^\mu \right) \sqcup \left( 
\displaystyle{\sqcup_{\shita{\mu \in \Lam_m}{\abs{\mu}\, is\, odd}}} GL_m(\calO_{k'})\cdot \pi^\mu\right), 
\end{eqnarray}
where $\abs{\mu} = \sum_{i=1}^m \mu_i$.

\bigskip
\begin{thm} \label{thm: G-orbits}
There are precisely two $G$-orbits in $X_n$ {\rm :}
\begin{eqnarray*} \label{G-orbits}
&&
G \cdot x_0 = \bigsqcup_{\shita{\lam \in \Lam_n^+}{\abs{\lam}\, is\, even}}\, K \cdot x_\lam, \quad 
G \cdot x_1 = \bigsqcup_{\shita{\lam \in \Lam_n^+}{\abs{\lam}\, is\, odd}}\, K \cdot x_\lam .
\end{eqnarray*}
where $\abs{\lambda}=\sum_{i=1}^n \lambda_i$,
$x_0 = 1_{2n+1}$ and $x_1 = \Diag(\pi, 1, \ldots, 1, \pi^{-1})$.
\end{thm}

\begin{proof}
First we see that there are at most two $G$-orbits in $X_n$ by Theorem \ref{thm: Cartan} and \eqref{eq:herm-G}.
We extend the $k$-automorphism $*$ of $k'$ to an element of $\Gamma = Gal(\ol{k}/k)$ and the action of $G$ on $X$ to $G(\ol{k})$ on $X(\ol{k})$, and write by the same symbol. We recall $X(\ol{k})$ is a single $G(\ol{k})$-orbit, and set
\begin{eqnarray*}
H(\ol{k}) = \set{h \in G(\ol{k})}{h \cdot 1_{2n+1} = 1_{2n+1}}.
\end{eqnarray*}
Then we obtain   
\begin{eqnarray*}
\lefteqn{H(\ol{k})}\\
& = &
\set{ \begin{pmatrix}
                     \frac12(A+B) & \lam c & \frac12(B-A)j_n\\
                     \lam^*d & f & \lam^* d j_n\\
                     \frac12 j_n(B-A)j_n & j_n\lam c &\frac12 j_n(A+B)j_n \end{pmatrix} }
{A \in U(1_n)(\ol{k}), \; \twomatrix{B}{c}{d}{f} \in U(1_{n+1})(\ol{k}) }\\[2mm]
& \cong &
U(1_n)(\ol{k}) \times U(1_{n+1})(\ol{k}),
\end{eqnarray*}
where $\lam \in \calO_{k'}^ \times$ such that $\lam\lam^* = \frac12$.
By the exact sequence of $\Gamma$-sets
$$
\begin{array}{lcccl}
1 \longrightarrow H(\ol{k}) \longrightarrow & G(\ol{k}) &\longrightarrow &X_n(\ol{k}) &\longrightarrow 1,
\\
{} & g & \longmapsto & g \cdot 1_{2n+1} & {}
\end{array}
$$
we have an exact sequence of pointed sets (cf.~\cite[I-\S 5.4]{Serre})
$$
1 \longrightarrow G \cdot 1_{2n+1} \longrightarrow X_n \longrightarrow H^1(\Gamma, H(\ol{k})) \stackrel{\eta}{\longrightarrow} H^1(\Gamma, G(\ol{k})),
$$
and it is known that $H^1(\Gamma, U(y)(\ol{k})) \cong C_2$ for any $y \in \calH_m(k'), \; m \geq 1$. Since $\eta$ is a map from $C_2 \times C_2$ to $C_2$, ${\rm Ker}(\eta)$ cannot be trivial and $G \cdot 1_{2n+1} \ne X_n$. Hence there are at least two $G$-orbits in $X_n$, thus exactly two $G$-orbits and they are given as above.
\end{proof}




\vspace{2cm}
\Section{Spherical function $\omega(x; s)$ on $X$}

{\bf 2.1.} 
We introduce a spherical function $\omega(x; s)$ on $X$ by Poisson transform from relative $B$-invariants. 
For a matrix $g \in G$, denote by $d_i(g)$ the determinant of lower right $i$ by $i$ block of $g$.
Then $d_i(x), \; 1 \leq i \leq n$ are relative $B$-invariants on $X$ associated with rational characters $\psi_i$ of $B$, where
\begin{eqnarray} \label{eq:rel inv, char} 
d_i(p \cdot x) = \psi_i(p)d_i(x), \quad \psi_i(p) = N_{k'/k}(d_i(p)), \quad (x \in X, \; p \in B).
\end{eqnarray}
We set
\begin{eqnarray} \label{eq:X-open}
X^{op} = \set{x \in X}{d_i(x) \ne 0, \; 1 \leq i \leq n}.
\end{eqnarray}
then $X^{op}(\ol{k})$ is a Zariski open $B(\ol{k})$-orbit, where $\ol{k}$ is the algebraic closure of $k$.  
For $x \in X$ and $s \in \C^n$, we consider the integral 
\begin{eqnarray} \label{def-sph}
\omega(x; s) = \int_{K}\, \abs{\bfd(k\cdot x)}^{s} dk, \quad \abs{\bfd(y)}^s = \prod_{i=1}^n\, \abs{d_i(y)}^{s_i},
\end{eqnarray}
where $dk$ is the normalized Haar measure on $K$, and $k$ runs over the set $\set{k \in K}{k\cdot x \in X^{op}}$.
The right hand side of (\ref{def-sph}) is absolutely convergent if $\real(s_i) \geq 0, \; 1 \leq i \leq n$, and continued to a rational function of $q^{s_1}, \ldots, q^{s_n}$, and we use the notation $\omega(x; s)$ in such sense.
We call $\omega(x;s)$ a spherical function on $X$, since it becomes an $\hec$-common eigenfunction on $X$ (cf.~\cite[\S 1]{JMSJ}, or \cite[\S 1]{French}). 
Indeed, $\hec$ is a commutative $\C$-algebra spanned by all the characteristic functions of double cosets $KgK, g \in G$ by definition,  
and we see 
\begin{eqnarray}
(f * \omega(\; ; s))(x) & \big( = & \int_{G}\, f(g)\omega(g^{-1}\cdot x; s)dg \big) \nonumber \\
&=&
\lam_s(f) \omega(x; s), \quad (f \in \hec),
\end{eqnarray}
where $dg$ is the Haar measure on $G$ normalized by $\int_{K}dg = 1$, and  $\lam_s$ is the $\C$-algebra homomorphism defined by 
\begin{eqnarray*} \label{eigen funct}
&&
\lam_s : \hec \longrightarrow \C(q^{s_1}, \ldots, q^{s_n}), \\
&&
\quad
f \longmapsto \int_{B}\, f(p)\abs{\psi(p)}^{-s}\delta(p) dp.
\end{eqnarray*}
Here $dp$ is the left-invariant measure on $P$ such that $\int_{P \cap K} dp = 1$ and $\delta(p)$ is the modulus character of $dp$ ($d(pq) = \delta(q)^{-1}dp$).

\bigskip
We introduce a new variable $z$ which is related to $s$ by
\begin{eqnarray} \label{change of var}
s_i = -z_i + z_{i+1} -1 + \frac{\pi\sqrt{-1}}{\log q} \quad (1 \leq i \leq n-1), \quad
s_n = -z_n -1 + \frac{\pi\sqrt{-1}}{2\log q}
\end{eqnarray}
and write $\omega(x;z) = \omega(x; s)$. 
We see
\begin{eqnarray}
\abs{\psi(p)}^s = (-1)^{v_\pi(p_1\cdots p_n)} \prod_{i=1}^n \abs{N_{k'/k}(p_i)}^{z_i} \delta^{\frac12}(p), \quad 
p \in B,
\end{eqnarray}
where $p_i$ is the $i$-the diagonal entry of $p$, $1 \leq i \leq n$. 
The Weyl group $W$ of $G$ with respect to the maximal $k$-split torus in $B$ acts on rational characters of $B$ as usual (i.e., $\sigma(\psi)(b) = \psi(n_\sigma^{-1}b n_\sigma)$ by taking a representative $n_\sigma$ of $\sigma$), so $W$ acts on  $z \in \C^n$ and on $s \in \C^n$ as well.  We will determine the functional equations of $\omega(x; s)$ with respect to this Weyl group action.
The group $W$ is isomorphic to $S_n \ltimes C_2^n$, $S_n$ acts on $z$ by permutation of indices, and $W$ is generated by $S_n$ and $\tau: (z_1, \ldots, z_n) \longmapsto (z_1, \ldots, z_{n-1}, -z_n)$. Keeping the relation (\ref{change of var}), we also write $\lam_z(f) = \lam_s(f)$.
Then the $\C$-algebra map $\lam_z$ is an isomorphism (the Satake isomorphism)
\begin{eqnarray} \label{satake iso}
\lam_z  &:& \hec \stackrel{\sim}{\longrightarrow} \C[q^{\pm 2z_1}, \ldots, q^{\pm 2z_n}]^W,
\end{eqnarray}
where the ring of the right hand side is the invariant subring of the Laurent polynomial ring $\C[q^{2z_1}, q^{-2z_1}, \ldots, q^{2z_n}, q^{-2z_n}]$ by $W$.

\bigskip
By using a result on spherical functions on the space of hermitian forms, we obtain the following result.

\begin{thm} \label{th: feq Sn}
The function $G_1(z) \cdot \omega(x; s)$ is invariant under the action of $S_n$ on $z$, where
\begin{eqnarray}
G_1(z) = \prod_{1 \leq i < j \leq n}\, \frac{1 + q^{z_i-z_j}}{1 - q^{z_i-z_j-1}}.
\end{eqnarray}
\end{thm}

\proof
By the embedding 
\begin{eqnarray*} \label{embed}
K_0 = GL_n(\calO_{k'})
\longrightarrow K, \quad h \longmapsto \wt{h} = \begin{pmatrix} {jh^{* -1}j} &  & \\
 & 1 & \\
 &  & {h} \end{pmatrix}, 
\end{eqnarray*}
and the normalized Haar measure $dh$ on $K_0$, we obtain, for $s \in \C^n$ satisfying $\real(s_i) \geq 0, \; 1 \leq i \leq n$,
\begin{eqnarray}
\omega(x;z) &=& \omega(x;s) = 
\dint{K_0}\,dh \dint{K}\, \abs{\bfd(k\cdot x)}^{s} dk
\nonumber \\
&=&
\dint{K_0}\,dh \dint{K}\, \abs{\bfd(\wt{h} k\cdot x)}^{s} dk
=
\dint{K}\, \dint{K_0}\, \abs{\bfd(\wt{h} k\cdot x)}^{s} dh dk\nonumber\\
&=&
\dint{K}\, {\zeta}_*^{(h)}(D(k\cdot x); s) dk. 
\label{id with herm}
\end{eqnarray}
Here $D(k\cdot x)$ is the lower right $n$ by $n$ block of $k \cdot x$ for $\set{k \in K}{k\cdot x \in X^{op}}$, and ${\zeta}_*^{(h)}(y; s)$ is a spherical function on $\calH_n(k')$ defined by
\begin{eqnarray*}
{\zeta}_*^{(h)}(y; s) = \dint{K_0}\, \abs{\bfd(h\cdot y)}^{s} dh, \quad (h \cdot y = hyh^*),
\end{eqnarray*}
where $h$ runs over the set $\set{h \in K_0}{d_i(h\cdot y) \ne 0, \; 1 \leq i \leq n}$. Then the assertion of Theorem~\ref{th: feq Sn} follows from the next proposition. 
We recall a similar spherical function on $\calH_n(k')$ and its functional equation, where we keep the definition of $G_1(z)$ and the relation between $s$ and $z$ as before.
\qed

%
\begin{prop} 
For any $y \in \calH_n(k')$, the function $G_1(z) \cdot {\zeta}_*^{(h)}(y; s)$ is holomorphic for $z \in \C^n$ and invariant under the action of $S_n$, where the relation between $z$ and $s$ is the same as in (\ref{change of var}).
\end{prop}

\proof
In \cite[\S 2]{JMSJ}, we have considered the following spherical function on $\calH_n(k')$ 
\begin{eqnarray*}
\zeta^{(h)}(y; s) = \dint{K_0}\, \prod_{i=1}^n\, \abs{\what{d}_i(h\cdot y)}^{s_i} dh, 
\end{eqnarray*}
where $\what{d}_i(y)$ is the determinant of upper left $i$ by $i$ block of $y$,
and shown that the function $G_1(z) \cdot \zeta^{(h)}(y; s)$ 
is holomorphic for $z \in \C^n$ and invariant under the action of $S_n$.
Here we note $z_i - z_j = -(s_i+\cdots +s_{j-1}) - (j-i)- (j-i)\frac{\pi\sqrt{-1}}{\log q}$ is determined by the relation of $s_1, \ldots, s_{n-1}$ in (\ref{change of var}). 
Since $d_i(y) = \det(y) \what{d}_{n-i}(y^{-1})$, we see in $s$-variable
\begin{eqnarray*} 
{\zeta}_*^{(h)}(y; s) &=& 
\dint{K_0}\, \abs{\det(h \cdot y)}^{\sum_{i=1}^n\,s_i} \prod_{i=1}^{n-1}\, \abs{\what{d}_{n-i}(h^{* -1}\cdot y^{-1})}^{s_i} dh\\
& = & 
\zeta^{(h)}(y^{-1}; s_{n-1}, \ldots, s_1, -(s_1+ \dots + s_n))\\
& \Big( & =  \zeta^{(h)}(y^{-1}; s'), say \Big).
\end{eqnarray*}
Let $w$ be the $z$-variable corresponding to the above $s'$ under the relation (\ref{change of var}). Then  $w_i - w_j = z_{n-j+1} - z_{n-i+1}$ for $1 \leq i < j \leq n$, and $G_1(w) = G_1(z)$. 
Hence $G_1(z) \cdot \zeta_*^{(h)}(y; s) = G_1(w) \cdot \zeta^{(h)}(y^{-1}; s')$ is holomorphic and $S_n$-invariant.
\qed

\slit
\noindent
{\bf 2.2.} In this subsection, we consider the case $n =1$ and show the following. 

\begin{prop} \label{explicit n=1}
Assume $n = 1$. For $x_\ell = \Diag(\pi^\ell, 1, \pi^{-\ell})$, $\ell \geq 0$, it holds
\begin{eqnarray*}
\omega(x_\ell; s)   
&=& 
\frac{(1+q^{-3-2s})}{(1+q^{-3})(1-q^{-4-4s})} \left\{q^{\ell s}(1-q^{-4-2s})-q^{-2(\ell+1)-\ell s}(1-q^{-2s})\right\}\\
&=&
\frac{\sqrt{-1}^\ell q^{-\ell}(1-q^{-1+2z})}{(1+q^{-3})(1+q^{2z})}
\left\{ q^{-\ell z}\frac{(1+q^{-2+2z})}{1-q^{2z}} + q^{\ell z}\frac{(1+q^{-2-2z})}{1-q^{-2z}} \right\},
\end{eqnarray*}
where $s = -z-1-\frac{\pi\sqrt{-1}}{2\log q}$, and for any $x \in X_1$
\begin{eqnarray*}
&&
\frac{1+q^{2z}}{1-q^{-1+2z}} \cdot \omega(x; z) \in \C[q^{z} + q^{-z}], \quad
\omega(x;z) = \frac{1-q^{-1+2z}}{q^{2z}-q^{-1}} \cdot \omega(x; -z).
\end{eqnarray*}
\end{prop}

\slit

\begin{lem}
\label{lem:vol}
For $\xi \in \calO_k^\times$ and $r \geq 0$, set $A(\xi;r) = \set{x \in \calO_{k'}}{v_\pi(N(x)-\xi) = r}$. By the Haar measure on $k'$ normalized by $vol(\calO_{k'}) = 1$, it holds
\begin{eqnarray*}
&&
vol(A(\xi; r)) = \left\{
\begin{array}{ll}
1-q^{-1}-q^{-2} & \mbox{if } r = 0,\\
(1-q^{-2})q^{-r} &\mbox{if } r > 0. 
\end{array} \right.
\end{eqnarray*}
\end{lem}

\proof
Assume $r > 0$. Since the norm map $N : \calO_{k'}^\times \longrightarrow \calO_k^\times$ is surjective, the induced map
$$
\ol{N} : \calO_{k'}^\times/{\rm mod}(\pi^{r+1}) \longrightarrow \calO_k^\times/{\rm mod}(\pi^{r+1}) 
$$
is surjective and $((q+1)q^r : 1)$-map. 
Thus we see the set
$$
A(\xi; r) = \calO_{k'}^\times \cap A(\xi; r)
$$
consists of $(q^2-1)q^r$ cosets ${\rm mod}(\pi^{r+1})$, hence $vol(A(\xi; r)) = (1-q^{-2})q^{-r}$. 
On the other hand, we see the set
$$
A(\xi; 0) = \pi\calO_{k'} \cup \set{x \in \calO_{k'}}{N(x) \not\equiv \xi \; {\rm mod}(\pi)},
$$
and it consists of $1 + \{(q^2-1)-(q+1)\} = q^2 - q-1$ cosets ${\rm mod}(\pi)$, hence $vol(A(\xi;0)) = 1-q^{-1}-q^{-2}$.
\qed

\slit
\textit{Proof of Proposition \ref{explicit n=1}.}
We recall Lemma~\ref{explicit K1}. It is easy to see that $vol(K_{1,1}) : vol(K_{1,2}) = 1 : q^{-3}$.
For $k \in K_{1,2}$ expressed as in Lemma~\ref{explicit K1}, we have  
\begin{eqnarray*}
\abs{d_1(k \cdot x_\ell)} &=& \abs{\pi^\ell N(c) + N(cd-b^*) + \pi^{-\ell} N(1+b^*d^*+cf)}\\
&=& \abs{\pi}^{-\ell}\abs{ \pi^{2\ell}N(c) + \pi^\ell N(cd-b^*) + N(1+b^*d^*+cf)} = \abs{\pi}^{-\ell}
\end{eqnarray*}
and
\begin{eqnarray} \label{K2}
\int_{K_{1,2}}\, \abs{d_1(k\cdot x_\ell)}^s dk = \frac{q^{-3 + \ell s}}{1+q^{-3}}.
\end{eqnarray}

\slit
For $k \in K_{1,1}$ expressed as in Lemma~\ref{explicit K1}, we have  
\begin{eqnarray*}  \label{K1-d1}
\abs{d_1(k\cdot x_\ell)} =  \abs{\pi^\ell + N(b) + \pi^{-\ell}(\frac{N(b)^2}{4}-c_0^2 \eps)} 
=
\abs{\pi}^{-\ell} \abs{(\pi^\ell + \frac{N(b)}{2})^2 - c_0^2 \eps },
\end{eqnarray*}
and
\begin{eqnarray*}
\int_{K_{1,1}}\abs{d_1(k\cdot x_\ell)}^s dk &=& 
\frac{q^{\ell s}}{1+q^{-3}} \int_{\calO_{k'}}db \int_{\calO_k}dc_0 \abs{(\pi^\ell + \frac{N(b)}{2})^2 - c_0^2 \eps }^s.
\end{eqnarray*}

\noindent
Assume $\ell$ is even and positive. Then
\begin{eqnarray*}
\lefteqn{(1+q^{-3})q^{-\ell s} \int_{K_{1,1}}\abs{d_1(k\cdot x_\ell)}^s dk} \nonumber \\
&=& 
\sum_{r = 0}^{\frac{\ell}{2}-1} \int_{\pi^r\calO_{k'}^\times} db \int_{\calO_k} dc_0 \abs{(\pi^\ell+\frac{N(b)}{2})^2-c_0^2\eps}^s +
 q^{-\ell}\int_{\calO_{k'}} db \int_{\calO_k} dc_0 \abs{\pi^{2\ell}(1+\frac{N(b)}{2})^2 - c_0^2 \eps}^s \nonumber\\
&=&
\sum_{r = 0}^{\frac{\ell}{2}-1} (1-q^{-2})q^{-2r} \left(\sum_{m=0}^{2r-1} (1-q^{-1})q^{-m-2ms} + q^{-2r-4rs} \right) \nonumber \\
&&
+
(1-q^{-1}-q^{-2})q^{-\ell} \left(  \sum_{m=0}^{\ell-1} (1-q^{-1})q^{-m-2ms} + q^{-\ell -2\ell s} \right) \nonumber\\
&&
+ 
\sum_{r \geq 1}(1 - q^{-2}) q^{-(\ell+r)} \left(  \sum_{m=0}^{\ell+r-1} (1-q^{-1})q^{-m-2ms} + q^{-(\ell+r) -2(\ell+r) s} \right),   \label{eq:even ell-1}
\end{eqnarray*}
where we have used Lemma \ref{lem:vol}. 
 Further, using the following equation
\begin{eqnarray} \label{useful}
\sum_{m=0}^{k-1} (1-q^{-1})q^{-m-2ms} + q^{-k-2ks} = \frac{(1-q^{-1}) + q^{-(k+1)-2ks}(1-q^{-2s})}{1 - q^{-1-2s}},
\end{eqnarray}
we have
\begin{eqnarray}
\lefteqn{(1-q^{-1-2s})(1+q^{-3})q^{-\ell s} \int_{K_{1,1}}\abs{d_1(k\cdot x_\ell)}^s dk} \nonumber \\
&=&
\sum_{r=0}^{\ell/2 -1} (1-q^{-2})q^{-2r} \left((1-q^{-1}) + (1-q^{-2s})q^{-(2r+1+4rs)} \right)
\nonumber \\
&&
+(1-q^{-1}-q^{-2})q^{-\ell} \left((1-q^{-1}) + (1-q^{-2s})q^{-(\ell+1+2\ell s)} \right) \nonumber\\
&&
+
\sum_{r\geq 1}\, (1-q^{-2})q^{-(\ell+r)} \left((1-q^{-1}) + (1-q^{-2s})q^{-(\ell+r+1+2(\ell+r)s)} \right)\nonumber\\
&=&
(1-q^{-1})(1-q^{-\ell}) + (1-q^{-2})q^{-1}(1-q^{-2s}) \frac{1-q^{-2\ell-2\ell s}}{1-q^{-4-4s}} \nonumber\\
&&
+(1-2q^{-1}+q^{-3})q^{-\ell} + (1-q^{-1}-q^{-2})q^{-(2\ell+1)-2\ell s} (1-q^{-2s})\nonumber\\
&&
+
(1-q^{-2})q^{-(\ell+1)} + (1-q^{-2})q^{-(2\ell +1)-2\ell s}(1-q^{-2s}) \frac{q^{-2-2s}}{1-q^{-2-2s}} \nonumber\\
&=&
1-q^{-1} + (1-q^{-2s})\nonumber \\
&&
\times \left(\frac{(q^{-1}-q^{-3})(1-q^{-2\ell -2\ell s})}{1-q^{-4-4s}}
+ (q^{-1}-q^{-2}-q^{-3})q^{-2\ell -2\ell s} + (q^{-1}-q^{-3})\frac{q^{-2(\ell+1)-2(\ell + 1)s}}{1-q^{-2-2s}}\right)\nonumber \\
&=&
\frac{1-q^{-1+2s}}{1-q^{-4-4s}} \left(1-q^{-3}+q^{-3-2s}-q^{-4-2s} - q^{-2(\ell+1)-2\ell s}(1-q^{-2s})(1+q^{-3-2s})\right). \label{eq:pos ell}
\end{eqnarray}
We note here that \eqref{eq:pos ell} holds for $\ell = 0$ by a direct calculation. 
Hence, we obtain for even $\ell$
\begin{eqnarray} 
&&
\int_{K}\abs{d_1(k\cdot x_\ell)}^s dk \nonumber \\
&=& \label{even}
\frac{(1+q^{-3-2s})q^{\ell s}}{(1+q^{-3})(1-q^{-4-4s})} \left\{(1-q^{-4-2s})-q^{-2\ell-2\ell s}(q^{-2}-q^{-2-2s})\right\}.
\end{eqnarray}
Changing the variable from $s$ to $z$ by the relation $s = -z-1+\frac{\pi\sqrt{-1}}{2\log q}$, we obtain
\begin{eqnarray}
&&
\int_{K} \abs{d_1(k\cdot x_\ell)}^s dk \nonumber \\
&=& 
\frac{{\sqrt{-1}}^\ell q^{-\ell}(1-q^{-1+2z})}{(1+q^{-3})(1-q^{4z})}
\left\{ q^{-\ell z}(1+q^{2z-2}) - q^{\ell z}(q^{-2}+ q^{2z}) \right\}\nonumber\\
&=& \label{even n=1}
\frac{\sqrt{-1}^\ell q^{-\ell}(1-q^{-1+2z})}{(1+q^{-3})(1+q^{2z})}
\left\{ q^{-\ell z}\frac{(1+q^{-2+2z})}{1-q^{2z}} + q^{\ell z}\frac{(1+q^{-2-2z})}{1-q^{-2z}} \right\}.
\end{eqnarray}

\slit
\noindent
Assume $\ell$ is odd. Then
\begin{eqnarray*}
\lefteqn{(1+q^{-3})q^{-\ell s} \int_{K_{1,1}}\abs{d_1(k\cdot x_\ell)}^s dk}\\
&=& 
\sum_{r = 0}^{\frac{\ell-1}{2}} \int_{\pi^r\calO_{k'}} db \int_{\calO_k} dc_0 \abs{(\pi^\ell+\frac{N(b)}{2})^2-c_0^2\ve}^s +
 q^{-(\ell+1)}\int_{\calO_{k'}} db \int_{\calO_k} dc_0 \abs{\pi^{2\ell}(1+\pi\frac{N(b)}{2})^2 - c_0^2 \ve}^s\\
&=&
\sum_{r = 0}^{\frac{\ell-1}{2}} (1-q^{-2})q^{-2r} \left(\sum_{m=0}^{2r-1} (1-q^{-1})q^{-m-2ms} + q^{-2r-4rs} \right) \\ 
&&
\qquad + 
q^{-(\ell+1)} \left(  \sum_{m=0}^{\ell-1} (1-q^{-1})q^{-m-2ms} + q^{-\ell -2\ell s} \right)
\quad(\mbox{by Lemma }\ref{lem:vol} \; )\\
\\
&=&
\frac{1}{1-q^{-1-2s}}
\Big{\{} \sum_{r=0}^{\frac{\ell-1}{2}} (1-q^{-2})q^{-2r} \left(1-q^{-1}+q^{-(2r+1)-2rs}(1-q^{-2s})\right) \\
&&
\qquad + 
q^{-(\ell+1)} \left(1-q^{-1}+q^{-(\ell+1)-2\ell s}(1-q^{-2s}\right)\Big{\}} \quad(\mbox{by } (\ref{useful}) \; )\\
&=&
\frac{1}{1-q^{-1-2s}} \Big{\{} (1-q^{-1}) + (q^{-1}-q^{-3})(1-q^{-2s})\frac{1-q^{2(\ell +1)(1+s)}}{1-q^{-4-4s}} + (1-q^{-2s})q^{-2(\ell+1)-2\ell s} \Big{\}}\\
&=&
\frac{1}{1-q^{-4-4s}} \Big{\{} (1-q^{-3} +q^{-3-2s}-q^{-4-2s}) + q^{-2(\ell+1)-2\ell s}(1-q^{-2s})(1+q^{-3-2s})  \Big{\}},
\end{eqnarray*}
hence we have
\begin{eqnarray*}
\lefteqn{\int_K \abs{d_1(k\cdot x_\ell)}^s dk}\\
&=&
\frac{(1+q^{-3-2s})q^{\ell s}}{(1+q^{-3})(1-q^{-4-4s})} \Big{\{} (1-q^{-4-2s}) + q^{-2\ell-2\ell s}(q^{-2}-q^{-2-2s}) \Big{\}}\\
&=& 
\frac{{\sqrt{-1}}^\ell q^{-\ell}(1-q^{-1+2z})}{(1+q^{-3})(1-q^{4z})}
\left\{ q^{-\ell z}(1+q^{2z-2}) - q^{\ell z}(q^{-2}+ q^{2z}) \right\}.
\end{eqnarray*}
Thus, we have the same identity (\ref{even n=1}) for odd  $\ell$ also. Hence we see 
\begin{eqnarray*}
\frac{1+q^{2z}}{1-q^{-1+2z}} \cdot \omega(x, z) \in \C[q^{z} + q^{-z}], 
\end{eqnarray*}
for any $x \in X_1$, which gives the functional equation with respect to $z$ and $-z$ (i.e. $s$ and $-s-2-\frac{\pi\sqrt{-1}}{\log q}$ in $s$-variable).\qed

\slit
\noindent
{\bf 2.3.}
We assume $n \geq 2$.
In this subsection, we give the functional equation of $\omega(x;s)$ for $\tau \in W$.

\begin{thm}   \label{th: tau}
For general size $n$, the spherical function satisfies the functional equation
$$
\omega(x; z) = \frac{1-q^{-1+2z_n}}{q^{2z_n}-q^{-1}} \omega(x; \tau(z)),
$$
where $\tau(z)=(z_1, \ldots, z_{n-1}, -z_n)$.
\end{thm}

\bigskip
We take the same strategy as in [HK].
Set 
$$
w_\tau = \begin{pmatrix}
1_{n-1} & & \\
& j_3 & \\
& & 1_{n-1}
\end{pmatrix},
$$
then the standard parabolic subgroup $P$ attached to $\tau$ is given as follows 
\begin{eqnarray} 
\lefteqn{P =  B \cup B w_\tau B } \nonumber \\
&=& \label{P-tau}
\set{
\begin{pmatrix}A &&\\& u & \\ && j_{n-1}A^{* -1}j_{n-1} \end{pmatrix}
\begin{pmatrix}1_{n-1} & \alp j_3 & B j_{n-1} \\& 1_3 & -\alp^*j_{n-1}\\ && 1_{n-1} \end{pmatrix}}
{\begin{array}{l}
A \in GL_{n-1}(k') ; \mbox{upper triangular}\\
u \in G_1= U(j_3), \alp \in M_{n-1,3}(k')\\
B \in M_{n-1}(k')\\
B + B^*  + \alp j_3\alp^*  = 0
\end{array} }.\nonumber\\
&& \label{P}
\end{eqnarray}
%
%
%
Here, $d_i(x), 1 \leq i \leq n-1$ are relative $P$-invariants, but $d_n(x)$ is not.
So we enlarge the group and space and consider the action of $P' = P \times GL_1(k')$ on $X' = X \times V$ with $V = M_{31}(k')$:
$$
(p,r) \star (x,v) = (p\cdot x, \rho(p)v r^{-1}), \quad (p,r) \in P', \; (x,v) \in X',
$$
where $\rho(p) = u \in U(j_3)$ for the decomposition of $p \in P$ as in (\ref{P-tau}).

We denote by $g_{(n+2)}$ the lower right $(n+2) \times (n+2)$-block of $g \in G (\subset GL_{2n+1})$. 
Set
\begin{eqnarray}
g(x, v) = \det\left[\twomatrixplus{v^*j_3}{}{}{1_{n-1}} \cdot x_{(n+2)}\right],
\end{eqnarray}
where the matrix inside of $[\quad]$ is of size $n$. 
Then, in the same way as in Lemma~2.5 in [HK], we see the following. 

\begin{lem}   \label{g(x,v)}
{\rm (1)} $g(x,v)$ is a relative $P'$-invariant associated by the character 
$$
P' \ni (p,r) \longmapsto N(d_{n-1}(p))N(r)^{-1} = \psi_{n-1}(p)N(r)^{-1},
$$
satisfying $g(x, v_0) = d_{n}(x)$ with $v_0 = {}^t(1,0,0)$. 

\mslit
{\rm (2)} For $x \in X^{op}$, there is $D_1(x) \in X_1$ such that
$$
g(x,v) = d_{n-1}(x) \cdot D_1(x)[v].
$$
\end{lem}

\bigskip
By using the embedding
\begin{eqnarray*}
K_1 = U(j_3) \hookrightarrow K = K_n, \; h \longmapsto \wt{h} = 
\begin{pmatrix}
1_{n-1} & & \\
& h & \\
& & 1_{n-1}
\end{pmatrix},
\end{eqnarray*}
we obtain the following identity 
\begin{eqnarray*}
\omega(x;s) &=& 
\int_{K}\, \prod_{i=1}^{n-2}\, \abs{d_i(k\cdot x)}^{s_i} \cdot \abs{d_{n-1}(k\cdot x)}^{s_{n-1}+s_n } \cdot \omega^{(1)}(D_1(k\cdot x); s_n) dk, 
\end{eqnarray*}
where $\omega^{(1)}(y; s_n)$ is the spherical function of size $n=1$. 
Then, by Proposition~\ref{explicit n=1}, we have
\begin{eqnarray*}
\lefteqn{\frac{1+q^{2z_n}}{1-q^{-1+2z_n}} \times \omega(x;z) =
\frac{1+q^{2z_n}}{1-q^{-1+2z_n}} \times \omega(x;s) }\\ 
&=& 
\int_{K}\, \prod_{i=1}^{n-1}\, \abs{d_i(k\cdot x)}^{s_i} \abs{d_{n-1}(k\cdot x)}^{s_n} \frac{1+q^{2z_n}}{q^{2z_n}-q^{-1}} \cdot \omega^{(1)}(D_1(k\cdot x); -s_n-2+\frac{\pi\sqrt{-1}}{\log q}) dk \nonumber \\
&=&
\frac{1+q^{-2z_n}}{1-q^{-1-2z_n}} \times 
\omega(x; s_1, \ldots, s_{n-2}, s_{n-1}+2s_n +2 + \frac{\pi\sqrt{-1}}{\log q}, -s_n-2+\frac{\pi\sqrt{-1}}{\log q})\\
&=&
\frac{1+q^{2z_n}}{q^{2z_n}-q^{-1}} \times 
\omega(x; \tau(z)),
\end{eqnarray*}
which completes the proof of Theorem~\ref{th: tau}.
\qed

\slit
\noindent
{\bf 2.4.}
We prepare some notation. Set
\begin{eqnarray*}
&&
\Sigma = \set{\pm e_i \pm e_j, \; 2e_i}{1 \leq i, j \leq n, \; i \ne j},\\
&&
\Sigma^+ = \set{e_i + e_j,  \; e_i - e_j}{1 \leq i < j \leq n} \cup \set{2e_i}{1 \leq i \leq n}
\end{eqnarray*}
where $e_i$ is the $i$-th unit vector in $\Z^n, \; 1 \leq i \leq n$. We note here that $\Sigma \cup \set{e_i}{1 \leq i \leq n}$ is the set of roots of $G$. 
We consider the pairing
\begin{eqnarray*}
\pair{t}{z} = \sum_{i=1}^n t_iz_i, \quad (t,z) \in \Z^n \times \C^n,
\end{eqnarray*}
which satisfies
\begin{eqnarray*}
\pair{\alp}{z} = \pair{\sigma(\alp)}{\sigma(z)}, \quad (\alp \in \Sigma, \; z \in \C^n, \; \sigma \in W).
\end{eqnarray*}

Then the following two theorems are proved in the same way as in \cite[Theorem 2.6]{HK} and \cite[Theorem 2.7]{HK}, based on Theorem~\ref{th: feq Sn} and Theorem~\ref{th: tau}.

\bigskip
\begin{thm} \label{th: feq}
The spherical function $\omega(x;z)$ satisfies the following functional equation
\begin{eqnarray} \label{Gamma-sigma}
\omega(x; z) = \Gamma_\sigma(z) \cdot \omega(x; \sigma(z)),
\end{eqnarray}
where 
\begin{eqnarray*} \label{sigma-factors}
\Gamma_\sigma(z) = \dprod{\alp \in \Sigma^+(\sigma)}\,
\frac{1 - q^{\pair{\alp}{z}-1}}{q^{\pair{\alp}{z}} - q^{-1}}, \quad
\Sigma^+(\sigma) = \set{\alp \in \Sigma^+_s}{-\sigma(\alp) \in \Sigma^+}. 
\end{eqnarray*}
\end{thm}

\bigskip
\begin{thm} \label{th: W-inv} 
The function $G(z) \cdot \omega(x; z)$ is holomorphic
 on $\C^n$ 
and $W$-invariant. In particular it is an element in $\C[q^{\pm z_1}, \ldots, q^{\pm z_n}]^{W}$, where
\begin{eqnarray*}
G(z) = \prod_{\alp \in \Sigma^+}\, \frac{1 + q^{\pair{\alp}{z}}}{1 - q^{\pair{\alp}{z}-1}}.
\end{eqnarray*}
\end{thm}

\vspace{2cm}
\Section{The explicit formula for $\omega(x;z)$}

{\bf 3.1.}
We give the explicit formula of $\omega(x; z)$. It suffices to show the explicit formula for each $x_\lam, \; \lam \in \Lam_n^+$ by Theorem~\ref{thm: Cartan},
since $\omega(x; z)$ is stable on each $K$-orbit.

\begin{thm} \label{th: explicit}
For $\lam \in \Lam_n^+$, one has the explicit formula
\begin{eqnarray*} 
\omega(x_\lam; z)
 &=&
\frac{(1+q^{-1})(1-q^{-2})^n}{w_{2n+1}(-q^{-1})} \cdot \frac{1}{G(z)} \cdot q^{\pair{\lam}{z_0}} \cdot Q_\lam(z),
\end{eqnarray*}
where $G(z)$ is given in Theorem~\ref{th: W-inv}, $z_0 \in \C^n$ is the value in $z$-variable corresponding to ${\bf0} \in \C^n$ in $s$-variable, 
\begin{eqnarray} 
&&
z_{0,i} = -(n-i+1)+(n-i+\frac12)\frac{\pi\sqrt{-1}}{\log q}, \quad 1 \leq i \leq n, \label{z0}\\
&&
w_m(t) = \prod_{i=1}^m (1 - t^i), \nonumber \\
&&
Q_\lam(z) = Q_\lam(z; -q^{-1}, -q^{-2}),\quad Q_\lam(z; t) = \sum_{\sigma \in W}\, \sigma\left(q^{-\pair{\lam}{z}} c(z, t) \right), \nonumber
\\
&& \label{c(z)}
c(z; t) = \prod_{\alp \in \Sigma^+}\, \frac{1 -t_\alp q^{\pair{\alp}{z}}}{1 - q^{\pair{\alp}{z}}}, \quad
t_\alp = \left\{ \begin{array}{ll} t_s = -q^{-1} & \textit{if } \alp = e_i \pm e_j\\
t_\ell = -q^{-2} & \textit{if } \alp =2e_i.  \end{array}\right. 
\end{eqnarray}

\end{thm}

\medskip
\begin{rem} \label{Rem 1}{\rm
We see the main part $Q_\lam(z)$ of $\omega(x; z)$ is contained in $\calR = \C[q^{\pm z_1}, \ldots, q^{\pm z_n}]^W$ by Theorem~\ref{th: W-inv}, and 
\begin{eqnarray} \label{P-Q}
P_\lam(z; t_s, t_\ell) = \frac{1}{W_\lam(t_s, t_\ell)} \cdot Q_\lam(z; t_s, t_\ell)
\end{eqnarray}
is a specialization of Hall-Littlewood polynomial of type $C_n$, where $W_\lam(t_s, t_\ell)$ is the Poincar{\' e} polynomial of the stabilizer subgroup $W_\lam$ of $W$ at $\lam$, and  
\begin{eqnarray} \label{w-lam}
&&
W_\lam(-q^{-1}, -q^{-2}) = \frac{\wt{w_\lam}(-q^{-1})}{(1+q^{-1})^{n+1}}, \\
&&
\wt{w_{\lam}}(t) = \left\{\begin{array}{ll}
w_{m_0(\lam)+1}(t) \cdot  \prod_{\ell \geq 0}\, w_{m_\ell(\lam)}(t) & \mbox{if }m_0 > 0 \\[2mm]
\prod_{\ell \geq 1}\, w_{m_\ell(\lam)}(t) & \mbox{if }m_0 = 0, 
\end{array} \right.
\quad 
m_\ell(\lam) = \sharp\set{i}{\lam_i = \ell}. \nonumber
\end{eqnarray}
Using $P_\lam(z) = P_\lam(z; -q^{-1}, -q^{-2})$ we have
\begin{eqnarray}
\omega(x_\lam; z) = \frac{(1-q^{-1})^n}{w_{2n+1}(-q^{-1})}\cdot \frac{1}{G(z)} \cdot \wt{w_\lam}(-q^{-1})\cdot q^{\pair{\lam}{z_0}}\cdot P_\lam(z), \quad (\lam \in \Lam_n^+).
\end{eqnarray}
It is known (cf. \cite{Mac}, \cite[Appendix B]{HK}) that the set 
$\set{P_\lam(z; t_s, t_\ell)}{\lam \in \Lam_n^+}$ forms an orthogonal $\C$-basis for $\calR$, in particular $P_{\bf0}(z) = 1$ (cf. (\ref{inner prod}) and  (\ref{P-lam-mu})). 
In the present case, the root system of the group $G = U(j_{2n+1})$ is of type $BC_n$, but we can write the explicit formula for $\omega(x;z)$ as above by using $P_\lam$ of type $C_n$. We need a different specialization from the case of unitary hermitian forms of even size, which is a homogeneous space of the group $U(j_{2n})$, where $(t_s, t_\ell) = (-q^{-1}, q^{-1})$  (cf. [HK]). 
} 
\end{rem} 

\begin{rem} \label{Rem 2}
{\rm
We give an interpretation of the constant $z_0$.
For $v\in\mathbb{Z}^n$, let
\begin{equation}
\label{eq:height}
  \mathbf{t}^{\Ht(v)}=\prod_{\beta\in\Sigma_s^+}t_s^{\langle v,\beta^\vee\rangle/2}
\prod_{\beta\in\Sigma_\ell^+}t_\ell^{\langle v,\beta^\vee\rangle/2},
\end{equation}
where $\beta^\vee=2\beta/\langle\beta,\beta\rangle$.
Then
for $v=\alpha\in\Sigma$, this coincides with the generalization of the heights of roots \cite{Mac2}.
On the other hand, \eqref{eq:height} can be rewritten as
\begin{equation*}
  \mathbf{t}^{\Ht(v)}=q^{\langle v,z_0\rangle},
\end{equation*}
where $z_0$ is given by \eqref{z0}. Thus $z_0$ can be regarded as a generalization of the dual Weyl vector.

From this viewpoint, the constant $z_0$ in the even case is calculated as
\begin{equation*}
  z_{0,i} = -(n-i+\frac12)+(n-i)\frac{\pi\sqrt{-1}}{\log q}, \quad 1 \leq i \leq n,
\end{equation*}
which corresponds to the change of variables
\begin{equation}
\label{eq:mod}
  s_n = -z_n -\frac12
\end{equation}
in \eqref{eq:chgv} with the same $s_i$ ($1\leq i\leq n-1$) as before.
This modification causes only the sign changes of $\omega(x_\lambda;z)$, that is,
$\omega(x_\lambda;z)$ with \eqref{eq:mod} 
 is the multiple by $(-1)^{|\lambda|}$  of the original $\omega(x_\lambda;z)$. In other words, on $G \cdot x_0$ the both coincide and on $G\cdot x_1$ the difference is the multiple by $-1$.
}
\end{rem}

\bigskip
We prove Theorem~\ref{th: explicit} by using a general expression formula of spherical functions given in \cite[\S 2]{French}. We have to check some assumptions, and it is easy to see them except the following condition (A3): 
  
\medskip
(A3) : \begin{minipage}[t]{14cm}
For any $y \in X \backslash X^{op}$, there exists some $\psi \in \langle{\psi_i \mid 1 \leq i \leq n}\rangle$ which is not trivial on the identity component of the stabilizer $B_y$.
\end{minipage}

\bigskip
\noindent
In the above statement, the character $\psi_i$ is given in \eqref{eq:rel inv, char} and the set $X^{op}$ is given in \eqref{eq:X-open}.
We admit (A3), which will be proved in \S 3.2, and prove the above theorem. According to the $B$-orbit decomposition
\begin{eqnarray*}
&&
X^{op} = \sqcup_{u \in \calU}\, X_u, \quad \calU = (\Z/2\Z)^n\\
&&
X_u = \set{x \in X^{op}}{ v_\pi(d_i(x)) \equiv u_1+\cdots +u_i \pmod{2}, \; 1 \leq i \leq n},
\end{eqnarray*}
we consider finer spherical functions
\begin{eqnarray*}
\omega_u(x; s) = \int_{K} \abs{\bfd(x)}_u^s dk, \quad \abs{\bfd(y)}_u^s = \left\{\begin{array}{ll}
\abs{\bfd(y)}^s & \mbox{if}\; y \in X_u,\\
0 & \mbox{otherwise}.
\end{array} \right.
\end{eqnarray*}
Then for each $\lam \in \Lam_n^+$ and generic $z$ we have the following identity:
\begin{eqnarray}
\left(\omega_u(x_\lam; z)\right)_{u \in \calU} = 
\frac{1}{Q} \sum_{\sigma \in W} \gamma(\sigma(z)) B(\sigma, z) \left(\delta_u(x;z) \right)_{u \in \calU},
\end{eqnarray}
where 
\begin{eqnarray*}
&&
Q = \sum_{\sigma \in W}[U\sigma U : U] \quad (U \; \mbox{is the Iwahori subgroup of $K$ associated with $B$}),\\ 
&&
\gamma(z) = \prod_{{\scriptsize \begin{array}{c}\alp = e_i \pm e_j\\1 \leq i < j \leq n\end{array}}}\,
\frac{1-q^{-2+2\pair{\alp}{z}}}{1-q^{2\pair{\alp}{z}}} \cdot 
\prod_{i=1}^{n} \frac{(1+q^{-2+2z_i})(1-q^{-1+2z_i}) }{1-q^{4z_i}},\\
&&
\delta_u(x_\lam;z) = \int_{U}\, \abs{\bfd(\nu \cdot x_\lam)}_u^s d\nu = \left\{\begin{array}{ll}
c_\lam q^{-\pair{\lam}{z}} & \mbox{if} \; x_\lam \in X_u\\
0 & \mbox{otherwise}, \end{array} \right.
\end{eqnarray*}
and $B(\sigma, z)$ is determined by the functional equation
\begin{eqnarray*}
\left(\omega_u(x_\lam; z)\right)_{u \in \calU} = B(\sigma, z) \left(\omega_u(x_\lam; \sigma(z))\right)_{u \in \calU},
\end{eqnarray*}
hence $B(\sigma, z)$ can be obtained by Theorem~\ref{th: feq}. 
We note here that $Q$ and $\gamma(z)$ are determined for the group $U(j_{2n+1})$ (cf. \cite[Theorem~4.4]{Car}) and $\gamma(z)$ coincides with $c(\lam)$ there for the character $\lam(p) = (-1)^{v_\pi(p_1\cdots p_n)} \prod_{i=1}^n \abs{N(p_i)}^{z_i}$, where $p_i$ is the $i$-th diagonal entry of $p \in B$. We don't need to calculate the constant $Q$ in advance, since it is determined by the property $\omega(x; s) \vert_{s = 0}\, = 1$ and $P_{\bf0}(z) = 1$.  
Thus, we obtain the explicit expression for $\omega(x_\lam,z) = \sum_{u \in \calU} \omega_u(x_\lam;z)$ in the similar line as in \cite[\S 3.1]{HK}.
\qed

\bigskip
We have the following immediately from Remark~\ref{Rem 1}. 

\begin{cor} \label{cor}
For $x_{\bf0} = 1_{2n+1}$, one has
$$
\omega(1_{2n+1}; z) = 
\frac{(1-q^{-1})^n w_n(-q^{-1}) w_{n+1}(-q^{-1})}{w_{2n+1}(-q^{-1})} \times
\frac{1}{G(z)}.
$$
\end{cor}

\vspace{1cm}
\noindent
{\bf 3.2.} In this subsection, we prove that $X_n$ satisfies the condition (A3). 
We set
\begin{eqnarray}
t(\bfb) = \Diag(b_1, \ldots, b_n, 1, b_n^{-1}, \ldots, b_1^{-1}) \in B_n, \quad \mbox{for \; } \bfb = (b_1, \ldots, b_n) \in k^{\times n}.
\end{eqnarray}
Assume $n =1$ and take any $x \in X_1$ such that $d_1(x) = 0$. Then, we see
$$
B_1\cdot x \ni y = \begin{pmatrix} 0 & 0 & 1 \\ 0 & -1 & 0\\ 1 & 0 & 0 \end{pmatrix}
$$
and $t(b)$ stabilizes $y$ for any $b \in k^\times$.

\begin{lem}  \label{Lsim3-4} 
Assume $x \in X_n, \; d_1(x) \ne 0$. Then $B \cdot x$ has an element of the shape
$$
\begin{pmatrix} c & 0& 0\\ 0 & y & 0\\ 0 & 0 & c^{-1} \end{pmatrix}, \quad  \; y \in X_{n-1}\backslash {X_{n-1}}^{op}, \; c = 1, \pi.
$$
\end{lem}
 
\proof
(This lemma is similar to \cite[Lemma~3.4]{HK}.) Since $d_1(x) \ne 0$, one can change by diagonal $(2n+1, 2n+1)$-entry of $x$ to be $1$ or $\pi^{-1}$.
Then one can change $(i, 2n+1)$-entry to be $0$ except $i = 2n+1$, and also for $(2n+1, i)$-entry. 
Then,we make all the $(i, 2n+1)$-component except $i = 2n+1$ to be $0$.
Then by the property $xj_{2n+1}xj_{2n+1} = 1_{2n+1}$, it has the shape as above.
\qed

\bigskip
The next lemma can be proved similarly to \cite[Lemma~3.5]{HK}.

\begin{lem} \label{Lsim3-5} 
Assume $x \in X_n$, $d_1(x)=0$, and the first non-zero entry in the last column from the bottom stands at $(2n-\ell+2, 2n+1)$-entry with $2 \leq \ell \leq n$. Then 
there is some $y \in B\cdot x$ which satisfies
\begin{eqnarray*}
\begin{array}{ll}
y_{1, j} = y_{j, 1} = \delta_{j, \ell}, & y_{2n+1, j} = y_{j, 2n+1} = \delta_{2n-\ell+2, j},\\
y_{\ell, j} = y_{j, \ell} = \delta_{j, 1}, & y_{2n-\ell+2, j} = y_{j, 2n-\ell+2} = \delta_{2n+1, j},
\end{array}
\end{eqnarray*}
and $B_y$ contains $t(\bfb)$ such that $b_1 = b, \; b_\ell = b^{-1}$ and other $b_j = 1$. 
\end{lem}

The next lemma follows from the property $x j_{2n+1}x j_{2n+1} = 1_{2n+1}$ for $x \in X$.

\begin{lem} 
Assume $x \in X_n$ satisfies $(2n-\ell+2, 2n+1)$-entry is $0$ for $1 \leq \ell \leq n$. Then $(n+1, 2n+1)$-entry is also $0$. 
\end{lem}

The next lemma can be proved similarly to \cite[Lemma~3.6]{HK}.

\begin{lem} \label{Lsim3-6} 
Assume $x \in X_n$, $d_1(x)=0$, and the first non-zero entry in the last column from the bottom stands at $(\ell, 2n+1)$-entry with $2 \leq \ell \leq n$. Then 
there is some $y \in B\cdot x$ which satisfies
\begin{eqnarray*}
\begin{array}{ll}
y_{1, j} = y_{j, 1} = \delta_{j, 2n-\ell+2}, & y_{2n+1, j} = y_{j, 2n+1} = \delta_{\ell, j},\\
y_{\ell, j} = y_{j, \ell} = \delta_{j, 2n+1}, & y_{2n-\ell+2, j} = y_{j, 2n-\ell+2} = \delta_{j, 1},
\end{array}
\end{eqnarray*}
and  $B_y$ contains $t(\bfb)$ such that $b_1 = b_\ell = b$ and other $b_j = 1$. 
\end{lem}

\bigskip
By Lemma~\ref{Lsim3-4} -- Lemma~\ref{Lsim3-6}, we have only to consider $x \in X$ of the following type:
$$
x = \left( \begin{array}{c|c|c}
* & * & {x_r}\\
      \hline
* & * & {0}\\
\hline
{x_r^*} & {0} & {0}
\end{array}\right), \quad
x_r = \begin{pmatrix} 
      * & {} & \xi_1\\ {} & 
\iddots
& {} \\ \xi_r & {} & {0}\end{pmatrix},\; 1 \leq r \leq n.
$$
Then we have $\xi_i = \pm 1, \; 1 \leq i \leq r$ and may change $x_r$ as $dj_r$ with $d = \Diag(\xi_1, \ldots, \xi_r)$ by the action of $B_n$. Hence we may assume
\begin{eqnarray}
x = \left( \begin{array}{c|c|c}
* & * & {dj_r}\\
      \hline
* & y & {0}\\
\hline
{j_r d} & {0} & {0}
\end{array}\right), \quad
1 \leq r \leq n, \; d = \Diag(\xi_1, \ldots, \xi_r),\; \xi_i = \pm 1, \; y = y^*.
\end{eqnarray}
If $d_1(y) \ne 0$, we may change $(r+1)$-th and $(2n-r+1)$-th rows and columns of $x$ into $0$ except diagonal elements, keeping the last $r$ rows and columns. Then, erasing those two rows and columns, we have an element in $X_{n-1} \backslash {X_{n-1}}^{op}$. 
If $y$ satisfies the assumption of Lemma~\ref{Lsim3-5} or Lemma~\ref{Lsim3-6}, we see $\psi_{r+1}({B_x}^0) \not\equiv 0$. 
Thus the remaining case is $r = n$ with the following shape: 
\begin{eqnarray}
x = \left( \begin{array}{c|c|c}
a & c & {dj_n}\\
      \hline
{}^tc & \xi & {0}\\
\hline
{j_n d} & {0} & {0}
\end{array}\right), \quad
d = \Diag(\xi_1, \ldots, \xi_n),\; \xi, \xi_i = \pm 1, \; c \in \{0, 1\}^n,
\end{eqnarray}
where we may assume $c$ as above by the action of a suitable diagonal element in $B_n$. By the property of  $\Phi_{xj_{2n+1}}(t)$ and $xj_{2n+1}xj_{2n+1} = 1_{2n+1}$, we have
\begin{eqnarray*}
&&
\sharp\set{\xi_i}{\xi_i = 1} = \left\{\begin{array}{ll}
\frac{n}{2}, & \mbox{$n$ is even, and}\; \xi = 1\\
\frac{n+1}{2}, & \mbox{$n$ is odd, and}\; \xi = -1 \end{array}\right.
\\
&&
c_i = 0 \; \; \mbox{if }\; \xi_i = \xi,\\[1mm]
&& \label{cond-a}
a_{ii} = \frac{\xi}{2} \; \; \mbox{if} \; c_i = 1, \quad a_{ij} = \frac{\xi}{2} \; \; \mbox{if} \; c_i = c_j = 1.
\end{eqnarray*}
We take a unipotent element $p \in B$ for which only non-zero entries outside of the diagonal are the following
\begin{eqnarray*}
\begin{array}{ll}
p_{i, n+1} = -p_{n+1, 2n-i+2} = -\frac{\ve}{2},  & \mbox{if } c_i = 1,\\
p_{i,n+1+i} = p_{i,n+1+k} = p_{k, n+1+i} = p_{i, n+1+k} -\frac18 & \mbox{if } c_i = c_k = 1, \; i \ne k.
\end{array}
\end{eqnarray*}
Then the $(n+1)$-th row and column of $p \cdot x$ is $0$ except $(n+1, n+1)$ which is $\ve$, and the new $a$-part satisfies 
$$
a_{ij} = a_{ji} = 0 \; \mbox{if }\; \xi_i = \xi_j.
$$
Then the stabilizer of $p\cdot x$ contains $t(\bfb)$ with $b_i = b$ if $\xi_i = \xi$ and $b_j = b^{-1}$ if $\xi_j = -\xi$, \; $b \in k^\times$.  
\qed

\vspace{2cm}
\Section{Spherical Fourier transform and Plancherel formula on $\SKX$}

We modify spherical function as follows:
\begin{eqnarray}
\Psi(x; z) = \omega(x; z) \big{/} \omega(1_{2n};z) \in \calR = \C[q^{\pm z_1}, \ldots, q^{\pm z_n}]^W,
\end{eqnarray}
and define the spherical Fourier transform on the Schwartz space  
\begin{eqnarray*}
\SKX &=& \set{\vphi: X \longrightarrow \C}{\mbox{left $K$-invariant, compactly supported}},
\end{eqnarray*}
by 
\begin{eqnarray} \label{def F-trans}
\begin{array}{lcll}
F :& \SKX &\longrightarrow & \calR\\
{} &  \vphi &\longmapsto &F(\vphi)(z) = \int_{X}\, \vphi(x)\Psi(x; z)dx,
\end{array}
\end{eqnarray}
where $dx$ is a $G$-invariant measure on $X$. The existence of a $G$-invariant measure is assured by the fact $X$ is a union of two $G$-orbits and $G$ is reductive, and we fix the normalization afterwards.
Then, for each characteristic function $ch_\lam$ of $K\cdot x_\lam, \; \lam \in \Lam_n^+$, we have
\begin{eqnarray} \label{F-lam}
F(ch_\lam)(z) = q^{\pair{\lam}{z_0}} \frac{ \wt{w_\lam}(-q^{-1})}{\wt{w_{\bf0}}(-q^{-1})} \cdot v(K \cdot x_\lam) P_\lam(z),
\end{eqnarray}
where $\wt{w_\lam}(-q^{-1})$ is defined in Theorem~\ref{th: explicit}, and $v(K\cdot x_\lam)$ is the volume of $K \cdot x_\lam$ with respect to $dx$.
On the other hand, we regard $\calR$ as an $\hec$-module through the Satake isomorphism $\lam_z$ (cf.~(\ref{satake iso})).
%
Then we have the following in the similar line as in Theorem~4.1 and Corollary 4.2 in \cite{HK}.

\begin{thm} \label{th: sph trans}
The spherical Fourier transform $F$ gives an $\hec$-module isomorphism 
\begin{eqnarray*}
\SKX \stackrel{\sim}{\rightarrow} \C[q^{\pm z_1}, \ldots, q^{\pm z_n}]^W (= \calR),
\end{eqnarray*}
where $\calR$ is regarded as $\hec$-module via $\lam_z$. Especially, $\SKX$ is a free $\hec$-module of rank $2^n$. 
\end{thm}

\begin{cor} \label{cor:4.2}
All the spherical functions on $X$ are parametrized by eigenvalues\\ $z \in \left( \C/\frac{2\pi\sqrt{-1}}{\log q} \Z \right)^n/W$ through $\hec \longrightarrow \C, \; f \longmapsto \lam_z(f)$.  
The set \\
$\set{\Psi(x; z + u)}{u \in \{0, \pi\sqrt{-1}/\log q \}^n }$ forms a basis of the space of spherical functions on $X$ corresponding to $z$.
\end{cor}

\bigskip
We introduce an inner product $\langle\;, \rangle_\calR$ on $\calR$ by
\begin{eqnarray} \label{inner prod}
\pair{P}{Q}_\calR = \int_{\fra^*} P(z) \ol{Q(z)} d\mu(z), \quad P, Q \in \calR.
\end{eqnarray}
Here 
\begin{eqnarray} \label{d-mu} 
\fra^* = \left\{ \sqrt{-1}\left( \R/\frac{2\pi}{\log q}\Z \right) \right\}^n,\quad
d\mu = \frac{1}{n!2^n} \cdot \frac{w_n(-q^{-1})w_{n+1}(-q^{-1})}{(1+q^{-1})^{n+1}} \cdot \frac{1}{\abs{c(z)}^2}dz,
\end{eqnarray}
where $dz$ is the Haar measure on $\fra^*$ with $\int_{\fra^*}dz = 1$ and $c(z) = c(z; -q^{-1}, -q^{-2})$ is defined in Theorem~\ref{th: explicit}.
Then, it is known (cf. \cite[Proposition~B.3]{HK}) that  
\begin{eqnarray} \label{P-lam-mu}
\pair{P_\lam}{P_\mu}_\calR = \pair{P_\mu}{P_\lam}_\calR = \delta_{\lam, \mu}\cdot 
\frac{ \wt{w_{\bf0}}(-q^{-1})}{\wt{w_\lam}(-q^{-1})}, \quad (\lam, \mu \in \Lam_n^+).
\end{eqnarray}
On the other hand, in the similar line to \cite[Lemma~4.4]{HK}, one has
\begin{eqnarray} \label{v-lam-mu}
\frac{v(K\cdot x_\lam)}{v(K\cdot x_\mu)} = \frac{q^{2\pair{\mu}{z_0}} \wt{w_\mu}(-q^{-1})}{q^{2\pair{\lam}{z_0}} \wt{w_\lam}(-q^{-1})}, \quad (\lam, \mu \in \Lam_n^+, \; \abs{\lam} \equiv \abs{\mu} ({\rm mod} 2)). 
\end{eqnarray}

\begin{thm}[Plancherel formula on $\SKX$] \label{th: Plancherel} 
Let $d\mu$ be the measure defined by (\ref{d-mu}). By the normalization of $G$-invariant measure $dx$ such that 
\begin{eqnarray} \label{vol K-lam}
v(K\cdot x_\lam) = q^{-2\pair{\lam}{z_0}}\frac{\wt{w_{\bf0}}(-q^{-1}) } {\wt{w_\lam}(-q^{-1})}, \qquad \lam \in \Lam_n^+,
\end{eqnarray}
one has for any $\vphi, \psi \in \SKX$
\begin{eqnarray} \label{the formula}
\int_{X} \vphi(x)\ol{\psi(x)} dx = \int_{\fra^*}F(\vphi)(z) \ol{F(\psi)(z)} d\mu(z).
\end{eqnarray}
\end{thm}

\medskip
{\it Outline of a proof.} 
Since $X$ decomposes into two $G$-orbits, we may normalize $dx$ on $X$ as
$$
v(K \cdot x_\lam) = \left\{\begin{array}{ll}
1 & \mbox{for } \lam = \bf0,\\
q^{-2n}\frac{(1-(-q^{-1})^{n})(1-(-q^{-1})^{n+1})}{1+q^{-1}} & \mbox{for } \lam = (1, 0, \cdots, 0)
\end{array}\right.,
$$
and the identity (\ref{vol K-lam}) follows from this and (\ref{v-lam-mu}). Then, for any $\lam, \mu \in \Lam_n^+$, we see by (\ref{F-lam}), (\ref{P-lam-mu}), and (\ref{vol K-lam})
$$
\int_{X} ch_\lam(x) \ol{ch_\mu(x)} dx = \delta_{\lam, \mu} q^{-2\pair{\lam}{z_0}} 
\frac{\wt{w_{\bf0}}(-q^{-1})}{\wt{w_\lam}(-q^{-1})} = \int_{\fra^*} F(ch_\lam)(z)\ol{F(ch_\mu)(z)} d\mu(z).
$$  
Since $\set{\ch_\lam}{\lam \in \Lam_n^+}$ spans $\SKX$, we see (\ref{the formula}) holds. 
\qed

\bigskip
The next corollary is an easy consequence of Theorem~\ref{th: Plancherel}.

\begin{cor}[Inversion formula] \label{cor: inv}
For any $\vphi \in \SKX$, 
$$
\vphi(x) = \int_{\fra^*} F(\vphi)(z) \Psi(x; z) d\mu(z), \quad x \in X.
$$
\end{cor}

\vspace{2cm}
\bibliographystyle{amsalpha}

\end{document}